\documentclass{amsart}
\usepackage{amsmath}
\usepackage{amssymb}
\usepackage{amsthm}
\usepackage{dsfont}
\usepackage{lineno}
\usepackage{subfigure}
\usepackage{graphicx}
\usepackage{multirow}
\usepackage{color}
\usepackage{xspace}
\usepackage{pdfsync}
\usepackage{hyperref} 
\usepackage{algorithm}
\usepackage{algorithmicx}
\usepackage{algpseudocode}
\usepackage{mathtools}
\usepackage{enumitem}
\usepackage{epstopdf}
\usepackage{bm}
\graphicspath{{Figures/}}

\DeclareMathOperator*{\argmin}{argmin}
\DeclareMathOperator*{\argmax}{argmax}

\DeclarePairedDelimiter\floor{\lfloor}{\rfloor}
\newcommand{\bq}{\begin{equation}}
\newcommand{\eq}{\end{equation}}
\newcommand{\R}{\mathbb{R}}
\newcommand{\Z}{\mathbb{Z}}
\newcommand{\abs}[1]{\left\vert#1\right\vert}

\newcommand{\G}{\mathcal{G}}
\newcommand{\bO}{\mathcal{O}}
\newcommand{\dist}{\text{dist}}

\newcommand{\Dt}{\mathcal{D}}

\newcommand{\Af}{\mathcal{A}}
\newcommand{\Sf}{\mathcal{S}}

\newcommand{\Zf}{\mathbb{Z}}
\newcommand{\Nf}{\mathbb{N}}
\newcommand{\xvec}{\mathbf{x}}
\newcommand{\MA}{Monge-Amp\`ere\xspace}

\newcommand{\la}[0]{\lambda}

\newcommand{\itan}[1]{\tan^{-1}\left({#1}\right)}

\newcommand{\x}[0]{\times}

\newcommand{\A}[0]{\forall}

\newcommand{\Or}{\mathcal{O}}

\newcommand{\lp}[0]{\left(}
\newcommand{\rp}[0]{\right)}

\newcommand{\figwidth}[0]{.7\textwidth} %for uniform scaling of figure sizes

\algnewcommand{\LineComment}[1]{\State \(\triangleright\) #1}

\newtheorem{theorem}{Theorem}

\theoremstyle{lemma}
\newtheorem{lemma}[theorem]{Lemma}

\newtheorem{cor}[theorem]{Corollary}
\newtheorem{definition}[theorem]{Definition}

\newtheorem{remark}[theorem]{Remark}

\newtheorem{hypothesis}[theorem]{Hypothesis}

\theoremstyle{remark}

\makeatletter
\newcommand\appendix@section[1]{%
\refstepcounter{section}%
\orig@section*{Appendix \@Alph\c@section: #1}%
}
\let\orig@section\section
\g@addto@macro\appendix{\let\section\appendix@section}
\makeatother

\begin{document}

\title[Finite difference methods for three-dimensional PDEs]{convergent finite difference methods for fully nonlinear elliptic equations in three dimensions}

\author{Brittany Froese Hamfeldt}
\address{Department of Mathematical Sciences, New Jersey Institute of Technology, University Heights, Newark, NJ 07102}
\email{bdfroese@njit.edu}
\author{Jacob Lesniewski}
\address{Department of Mathematical Sciences, New Jersey Institute of Technology, University Heights, Newark, NJ 07102}
\email{jl779@njit.edu}

\thanks{The first author was partially supported by NSF DMS-1619807 and NSF DMS-1751996. The second author was partially supported by NSF DMS-1619807.}

\begin{abstract}
%Type Abstract Here:
We introduce a generalized finite difference method for solving a large range of fully nonlinear elliptic partial differential equations in three dimensions.  Methods are based on Cartesian grids, augmented by additional points carefully placed along the boundary at high resolution.  We introduce and analyze a least-squares approach to building consistent, monotone approximations of second directional derivatives on these grids.  We then show how to efficiently approximate functions of the eigenvalues of the Hessian through a multi-level discretization of orthogonal coordinate frames in $\R^3$.  The resulting schemes are monotone and fit within many recently developed convergence frameworks for fully nonlinear elliptic equations including non-classical Dirichlet problems that admit discontinuous solutions, Monge-Amp\`ere type equations in optimal transport, and eigenvalue problems involving nonlinear elliptic operators.  Computational examples demonstrate the success of this method on a wide range of challenging examples.

\end{abstract}

\date{\today}    
\maketitle

\section{Introduction}\label{sec:intro}

In this article, we introduce and implement a convergent finite difference method for solving a large class of fully nonlinear elliptic partial differential equations (PDEs) on general three-dimensional domains.  
The method we develop encompasses a range of challenging problems including Pucci's maximal and minimal equations, obstacle problems, prescribed curvature equations, \MA type equations arising in optimal transport, and eigenvalue problems involving nonlinear PDEs.  

\subsection{Background}
Fully nonlinear elliptic partial differential equations (PDEs) appear in a variety of applications including optimal transport, seismology \cite{EF_FWI}, astrophysics \cite{2002Nature}, mathematical finance \cite{fleming2006controlled}, materials science~\cite{Hoffman}, and molecular engineering~\cite{Bates}. These problems are challenging because they often include discontinuous or sharp jumps in the data, involve intricate domains, and may require data to satisfy a solvability condition that is not known \emph{a priori}.

In recent years, the numerical solution of these equations has received a great deal of attention, and several new methods have been developed including finite difference methods~\cite{BFO_MA,FinnGrid,Loeper,Saumier,SulmanWilliamsRussell}, finite element methods~\cite{Awanou,Bohmer,BrennerNeilanMA2D,Smears}, least squares methods~\cite{DGnum2006}, and methods involving fourth-order regularization terms~\cite{FengNeilan}.  However, these methods are not designed to compute weak solutions. When the ellipticity of the equation is degenerate or no smooth solution exists, methods become very slow, are unstable, or converge to an incorrect solution.

A couple convergence frameworks have emerged in recent years. The approach of~\cite{FengLewis} introduces the concept of generalized monotonicity, similar in flavor to the fourth-order regularization of~\cite{FengNeilan}, to produce convergent methods for a class of Hamilton-Jacobi-Bellman equations.
Another powerful framework, which informs the method described in the present article, is the approach of Barles and Souganidis~\cite{BSnum}, which shows that consistent, monotone methods converge if the limiting PDE satisfies a comparison principle.  A variety of methods have been developed within this framework~\cite{FinlayOberman,FroeseMeshfreeEigs,FO_MATheory,ObermanEigenvalues,mirebeau2015MA}.  Moreover, these convergence proofs have recently been extended to non-classical Dirichlet problems~\cite{Hamfeldt_Gauss}, optimal transport type boundary conditions~\cite{Hamfeldt_BVP2}, and eigenvalue problems involving nonlinear PDEs~\cite{HL_lagrangian}.

\subsection{Contributions of this work}
The method we describe applies to nonlinear equations that depend on various second directional derivatives,
\bq\label{eq:PDE1}
F(\xvec,u(\xvec),u_{\bm{\nu}\bm{\nu}}(\xvec); \bm{\nu}\in\Af\subset\mathbb{S}^2) = 0,
\eq
where the admissible set $\Af$ is used to characterize a finite set of unit vectors in $\R^3$.  We also consider functions of the eigenvalues $\lambda_1, \lambda_2, \lambda_3$ of the Hessian matrix,
\bq\label{eq:PDE2}
F(\lambda_1(D^2u(\xvec)),\lambda_2(D^2u(\xvec)),\lambda_3(D^2u(\xvec))) \equiv G\left(\sum\limits_{j=1}^3\phi(\lambda_j(D^2u(\xvec)))\right)= 0,
\eq
where $\phi$ is a concave function and $G$ is non-increasing and continuous.  We note that this encompasses a range of elliptic operators including PDEs of \MA type and various curvature equations.
Finally, we consider eigenvalue problems of the form
\bq\label{eq:PDE3}
F(\xvec,D^2u(\xvec)) = c
\eq
where both the function $u$ and the constant $c\in\R$ are unknown.
Moreover, we show how to enforce a range of boundary conditions including Dirichlet conditions, Neumann conditions, and the second type boundary condition $\nabla u(\Omega_1) \subset \bar{\Omega}_2$ arising in optimal transport.

The starting point of our method is the generalized finite difference methods of~\cite{FroeseMeshfreeEigs}, which produced consistent, monotone schemes for a large class of nonlinear elliptic equations using very general two-dimensional point clouds.  While much of the convergence theory applies to general bounded domains in $\R^n$, the transition to three dimensions introduces several new challenges that are not present in two dimensions. 

A first challenge that we face is the discretization of the boundary of three-dimensional domains. In order to preserve both consistency and monotonicity in the entire domain, it is necessary to over-resolve the boundary in a precise way.  This precludes the use of many standard structure grids such as Cartesian meshes.  We describe a new approach to discretizing our domains that preserves a great deal of structure (which is needed for efficient evaluation of the nonlinear operators), while fitting within the precise requirements needed to construct monotone schemes.

A second challenge is the approximation of second directional derivatives.  In two dimensions, this can be accomplished explicitly even on very complicated meshes.  In three dimensions, explicit formulas are no longer possible in general.  Instead, we describe an optimization approach that provably yields a consistent, monotone approximation.

A third major challenge is discretizing general functions of the eigenvalues of the Hessian.  In two dimensions, these two eigenvalues can be represented as the maximum and minimum possible second directional derivatives.  Three dimensions introduces a third eigenvalue, and this simple approach does not easily generalize in a way that preserves monotonicity.  We introduce an alternate approach that allows for monotone approximation of a large range of functions of the eigenvalues of the Hessian.

Finally, we note that the shift to a higher dimension brings the practical concerns of memory and processing speed to the forefront.  A naive implementation of our proposed discretization would be computationally intractable.  Instead, we introduce techniques derived from the structure of the underlying problems and approximations to describe an efficient method for producing a consistent, monotone approximation of many fully nonlinear elliptic PDEs in three dimensions. 

\subsection{Contents}
In \autoref{sec:background}, we review the theory of generalized finite difference approximations for fully nonlinear elliptic equations. In \autoref{sec:numerics}, we describe the three-dimensional discretization of~\eqref{eq:PDE1}-\eqref{eq:PDE3}.  In \autoref{sec:implementation}, we discuss practical considerations relating to the efficient construction of our discretization.  In \autoref{sec:results}, we provide computational results for a large range of challenging problems.  Finally, in \autoref{sec:conclusion}, we provide conclusions and perspective.

\section{Approximation of elliptic equations}\label{sec:background}
%Background:  (starting point is background from the other paper, but with eigenvalue specific section removed)
In this section, we briefly review relevant results on the construction and convergence of numerical methods for solving fully nonlinear elliptic equations.

\subsection{Elliptic equations}
The PDE operators we consider in this work are degenerate elliptic.
\bq\label{eq:PDEelliptic} F(\xvec,u(\xvec),D^2u(\xvec)) = 0, \quad \xvec\in\bar{\Omega}\subset\R^3.\eq
\begin{definition}[Degenerate elliptic]\label{def:elliptic}
The operator
$F:\bar{\Omega}\times\R\times\Sf^2\to\R$
is \emph{degenerate elliptic} if 
\[ F(\xvec,u,X) \leq F(\xvec,v,Y) \]
whenever $u \leq v$ and $X \geq Y$.
\end{definition}
We note that the definition of the operator is extended onto the boundary of the domain, and includes the relevant boundary conditions.

The PDE operators~\eqref{eq:PDE1}-\eqref{eq:PDE2} that we consider in this work are degenerate elliptic if they are non-decreasing functions of the argument $u$ and non-increasing functions of all subsequent arguments (which involve second directional derivatives).

Since degenerate elliptic equations need not have classical solutions, solutions need to be interpreted in a weak sense.  The numerical methods developed in this article are guided by the very powerful concept of the viscosity solution~\cite{CIL}.

\begin{definition}[Upper and lower semi-continuous envelopes]\label{def:envelope}
The \emph{upper and lower semi-continuous envelopes} of a function $u(\xvec)$ are defined, respectively, by
\[ u^*(\xvec) = \limsup_{\mathbf{y}\to \xvec}u(\mathbf{y}), 
\quad u_*(\xvec) = \liminf_{\mathbf{y}\to \xvec}u(\mathbf{y}). \]
\end{definition}

\begin{definition}[Viscosity solution]\label{def:viscosity}
An upper (lower) semi-continuous function $u$ is a \emph{viscosity subsolution (supersolution)} of~\eqref{eq:PDEelliptic} if for every $\phi\in C^2(\bar{\Omega})$, whenever $u-\phi$ has a local maximum (minimum)  at $\xvec \in \bar{\Omega}$, then
\[ 
F_*^{(*)}(\xvec,u(\xvec),D^2\phi(\xvec)) \leq (\geq)  0 .
\]
A function $u$ is a \emph{viscosity solution} of~\eqref{eq:PDEelliptic} if $u^*$ is a subsolution and $u_*$ a supersolution.
\end{definition}

An important property of many elliptic equations is the comparison principle, which immediately implies uniqueness of the solution.
\begin{definition}[Comparison principle]\label{def:comparison}
A PDE has a \emph{comparison principle} if whenever $u$ is an upper semi-continuous subsolution and $v$ a lower semi-continuous supersolution of the equation, then $u \leq v$ on $\bar{\Omega}$.
\end{definition}

\subsection{Approximation}\label{sec:discBackground}
In order to construct convergent approximations of elliptic operators, we will rely on the framework introduced by Barles and Souganidis~\cite{BSnum} and further extended by Oberman~\cite{ObermanEP} and the authors of this work~\cite{Hamfeldt_Gauss,Hamfeldt_BVP2,HL_lagrangian}.

We consider finite difference schemes that have the form
\bq\label{eq:approx} F^h(\xvec,u(\xvec),u(\xvec)-u(\cdot)) = 0 \eq
where $h$ is a small parameter relating to the grid resolution.

The convergence framework requires notions of consistency and monotonicity, defined below.

\begin{definition}[Consistency]\label{def:consistency}
The scheme~\eqref{eq:approx} is \emph{consistent} with the equation~\eqref{eq:PDEelliptic} if for any smooth function $\phi$ and $x\in\bar{\Omega}$,
\[ \limsup_{\epsilon\to0^+,\mathbf{y}\to \xvec,\xi\to0} F^h(\mathbf{y},\phi(\mathbf{y})+\xi,\phi(\mathbf{y})-\phi(\cdot)) \leq F^*(\xvec,\phi(\xvec),\nabla\phi(\xvec),D^2\phi(\xvec)), 
\]
\[ \liminf_{\epsilon\to0^+,\mathbf{y}\to \xvec,\xi\to0} F^h(\mathbf{y},\phi(\mathbf{y})+\xi,\phi(\mathbf{y})-\phi(\cdot)) \geq F_*(\xvec,\phi(\xvec),\nabla\phi(\xvec),D^2\phi(\xvec)). \]
\end{definition}

\begin{definition}[Monotonicity]\label{def:monotonicity}
The scheme~\eqref{eq:approx} is monotone if $F^h$ is a non-decreasing function of its final two arguments.
\end{definition}

Schemes that satisfy these two properties respect the notion of the viscosity solution at the discrete level.  In particular, these schemes preserve the maximum principle and are guaranteed to converge to the solution of the underlying PDE under a range of interesting settings.

Another important property of schemes is stability, which allows discrete solutions to be bounded uniformly.
\begin{definition}[Stability]\label{def:stability}
The scheme~\eqref{eq:approx} is stable if there exists some $M\in\R$ such that if $u^h$ is any solution of~\eqref{eq:approx} then $\|u^h\|_\infty \leq M$.
\end{definition}
Under mild conditions relating to the well-posedness of the limiting PDE and continuity of~\eqref{eq:approx} in its final two arguments, consistent and monotone schemes are automatically stable~\cite[Lemmas~35-36]{Hamfeldt_Gauss}.

The convergence framework of Barles and Souganidis applies to PDEs that satisfy a comparison principle.
\begin{theorem}[Convergence~\cite{ObermanEP}]\label{thm:convergeVisc}
Let $u$ be the unique viscosity solution of the PDE~\eqref{eq:PDEelliptic}, where $F$ is a degenerate elliptic operator with a comparison principle.  Let the approximation $F^h$ be consistent, monotone, and stable and $u^h$ any solution of the scheme~\eqref{eq:approx}.  Then $u^h$ converges uniformly to $u$ as $h\to0$.
\end{theorem}

Recently, convergence results have also been obtained for a variety of equations that do not have a traditional comparison principle.  This includes non-classical Dirichlet problems that admit discontinuous solutions~\cite{Hamfeldt_Gauss}, \MA type equations equipped with optimal transport type boundary constraints~\cite{Hamfeldt_BVP2}, and eigenvalue problems of the form~\eqref{eq:PDE3} that require solving for the solution of both a fully nonlinear PDE and an unknown scalar constant~\cite{HL_lagrangian}.

\subsection{Generalized Finite Difference Methods}\label{sec:genFD}
The results discussed in the previous section provide a powerful tool for proving convergence of monotone schemes.  However, they do not offer insight into the actual construction of monotone methods.  Indeed, it is well known that there are linear elliptic operators for which no finite difference stencil of bounded width admits a consistent, monotone discretization~\cite{Kocan, MotzkinWasow}.  

Wide stencil methods on Cartesian grids have been developed for a variety of fully nonlinear elliptic PDEs~\cite{mirebeau2015MA,ObermanEigenvalues}.  However, these can fail to preserve both consistency and monotonicity at points near the boundary, particularly in the absence of Dirichlet boundary conditions.
For insight into the construction of globally consistent and monotone schemes on general domains, we turn to the two-dimensional meshfree finite difference methods developed in~\cite{FroeseMeshfreeEigs}.

We begin by introducing some notation, which applies equally well to higher dimensional settings.
\begin{enumerate}
\item[(N1)] $\Omega\subset\R^d$ is a bounded domain with Lipschitz boundary $\partial\Omega$. 
\item[(N2)] $\G\subset\bar{\Omega}$ is a point cloud consisting of the points $\xvec_i$, $i=1,\ldots,N$.
\item[(N3)] $h = \sup\limits_{\xvec\in{\Omega}}\min\limits_{\mathbf{y}\in\G}\abs{\xvec-\mathbf{y}}$ is the spatial resolution of the point cloud.  In particular, every ball of radius $h$ contained in $\bar{\Omega}$ contains at least one discretization point $\xvec_i$.
\item[(N4)] $h_B = \sup\limits_{\xvec\in{\partial\Omega}}\min\limits_{\mathbf{y}\in\G\cap\partial\Omega}\abs{\xvec-\mathbf{y}}$ is the resolution of the point cloud on the boundary.  In particular, every ball of radius $h_B$ centered at a boundary point $\xvec\in\partial\Omega$ contains at least one discretization point $\xvec_i \in \G\cap\partial\Omega$ on the boundary.
\item[(N5)] $\delta = \min\limits_{\xvec\in\Omega\cap\G^h}\inf\limits_{\mathbf{y}\in\partial\Omega}\abs{\xvec-\mathbf{y}}$ is the distance between the set of interior discretization points and the boundary.  In particular, if $\xvec_i\in\G^h\cap\Omega$ and $\xvec_j\in\partial\Omega$, then the distance between $\xvec_i$ and $\xvec_j$ is at least $\delta$.
%\item[(N6)] $d\theta$ is the angular resolution used to approximate the second directional derivatives $u_{\nu\nu}$.
%\item[(N6)] $d\theta$ is the angular resolution used to approximate the nonlinear operator.
\item[(N6)] $\epsilon$ is a search radius associated with the point cloud.
\end{enumerate}

Using this notation, we can pose some necessary hypotheses on the point cloud and related discretization parameters.
\begin{hypothesis}[Hypotheses on point cloud]\label{hyp:grid}
We require the discretization $\G$ of $\bar{\Omega}$ to satisfy:
\begin{enumerate}
\item[(H1)] The boundary resolution satisfies $h_B/\delta\to0$ as $h\to0$.
\item[(H2)] The search radius satisfies both $\epsilon\to0$ and $h/\epsilon\to0$ as $h\to0$.
\end{enumerate}
\end{hypothesis}

Suppose we wish to approximate the second directional derivative $u_{\bm{\nu}\bm{\nu}}$ at some point $\xvec_0\in\G^h\cap\Omega$.  We begin by considering as candidate neighbors all nodes within a distance $r$ of $\xvec_0$.  From this we select four points $\xvec_1,\xvec_2,\xvec_3,\xvec_4 \in \G^h\cap B(\xvec_0,\epsilon)$, one in each of the four quadrants defined by the axes $\xvec_0+\bm{\nu} t$ and $\xvec_0 + \bm{\nu}^\perp t$, that are as well-aligned as possible with the line $\xvec_0 + \bm{\nu} t$.  See Figure~\ref{fig:stencil2D}.

\begin{figure}[htp]
\centering
\subfigure[]{
\includegraphics[width=0.4\textwidth]{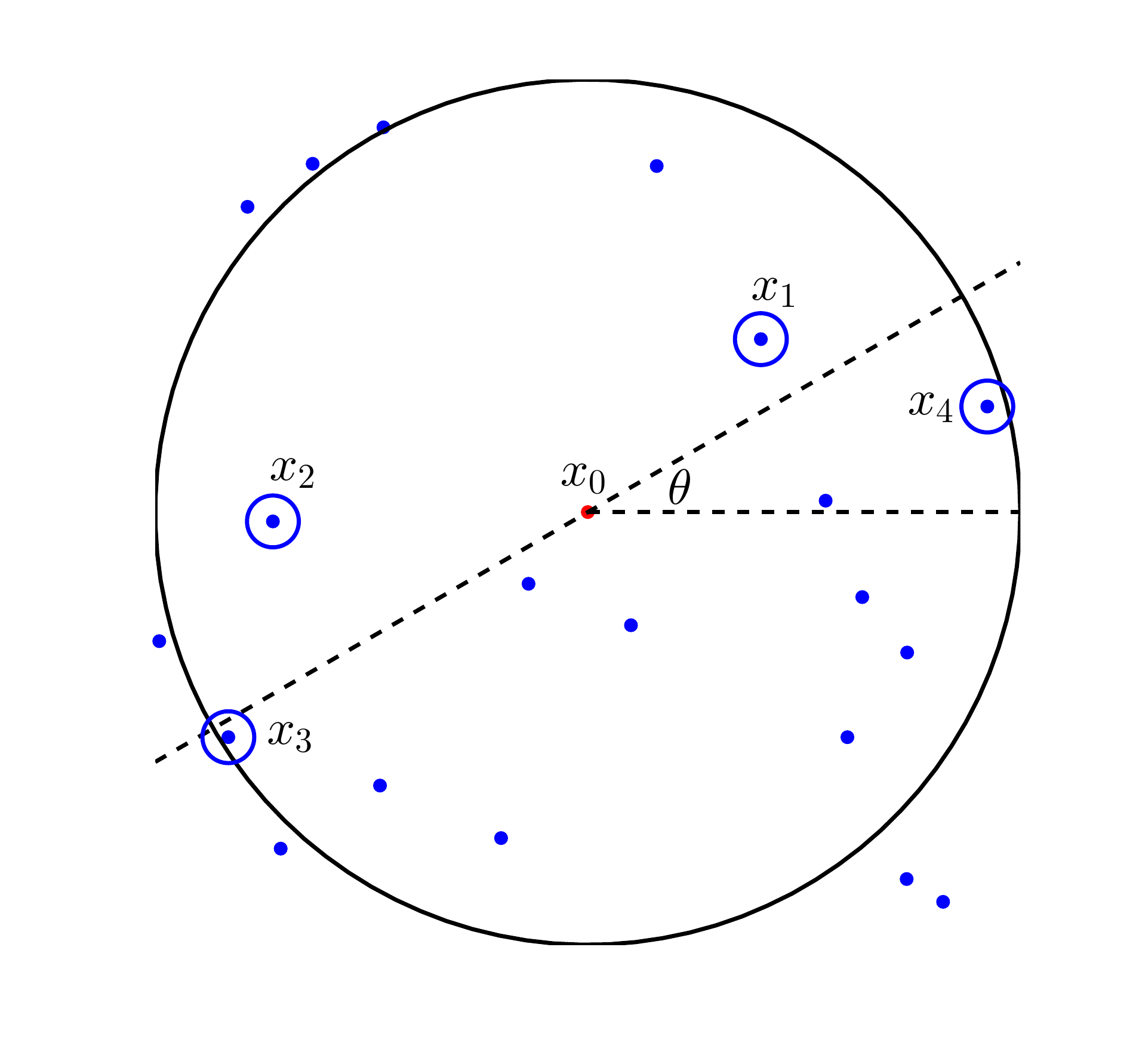}\label{fig:stencil1}}
\subfigure[]{
\includegraphics[width=0.47\textwidth]{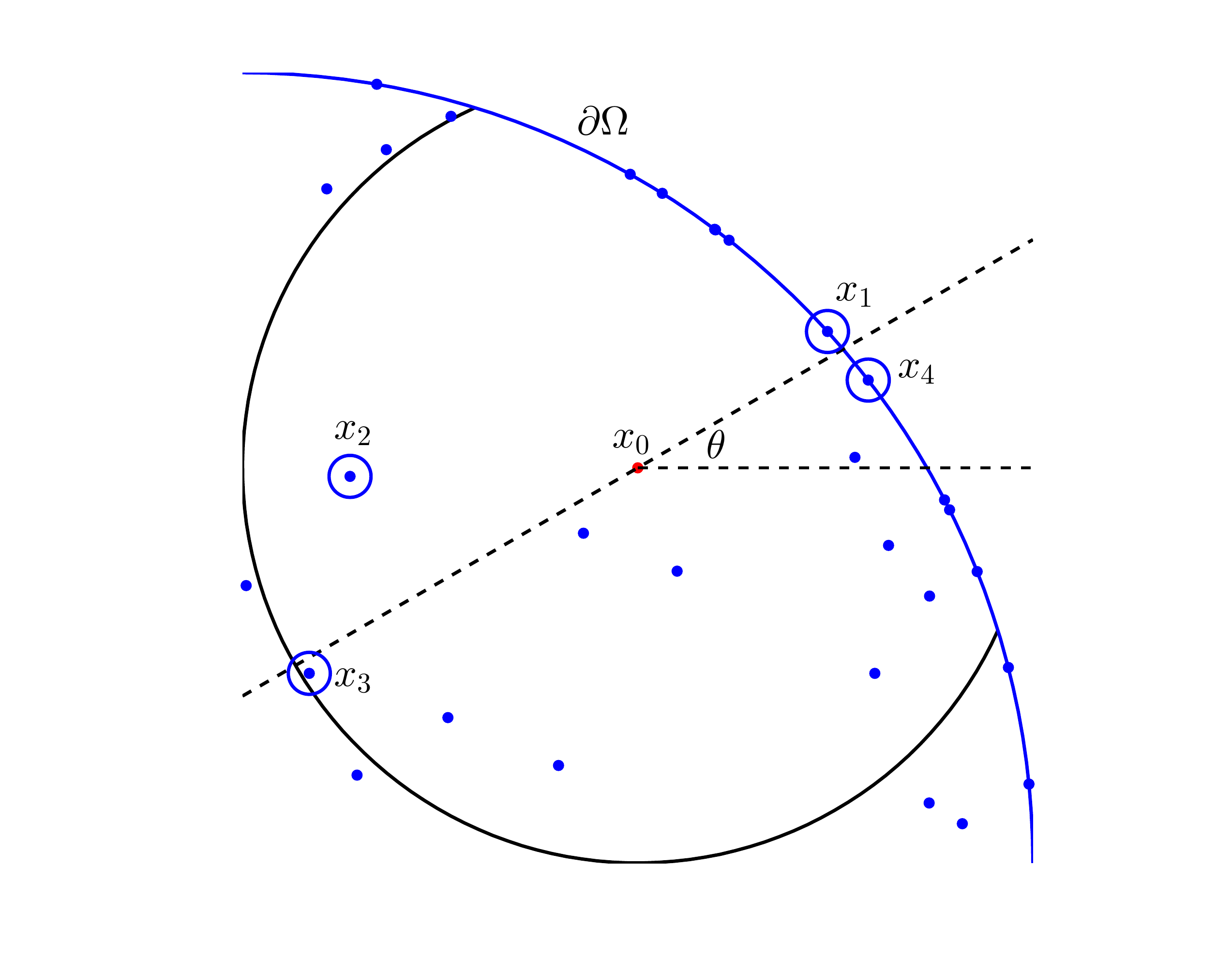}\label{fig:stencil2}}
\caption{A finite difference stencil chosen from a point cloud \subref{fig:stencil1}~in the interior and \subref{fig:stencil2}~near the boundary.}
\label{fig:stencil2D}
\end{figure}

Using these four neighbors, we seek an approximation of the form
\[ -u_{\bm{\nu}\bm{\nu}}(\xvec_0) \approx -\Dt_{\bm{\nu}\bm{\nu}}^hu(\xvec_0) = -\sum\limits_{j=1}^4 a_j(u(\xvec_j)-u(\xvec_0)). \]
Here monotonicity requires that each of the $a_j \geq 0$.  In two dimensions, it is possible to find an explicit form of the coefficients $a_j$ that yields a consistent, monotone approximation under Hypothesis~\ref{hyp:grid}.

We remark that the discretization error of the resulting scheme depends on two parameters: the effective spatial resolution $\epsilon$ (the maximum distance between $\xvec_0$ and the points $\xvec_j$ used in the finite difference stencil) and the angular resolution $d\phi$ (the maximum angle between the axis aligned with $\bm{\nu}$ and the vector $\xvec_j-\xvec_0$).  These components of the error are not independent of each other, and we find that the angular resolution is bounded by $d\phi = \max\{\bO(h/\epsilon),\bO(h_B/\delta)\}$.  See Figure~\ref{fig:MeshfreeAngular}.

\begin{figure}
\centering
\subfigure[]{\includegraphics[width=0.45\textwidth]{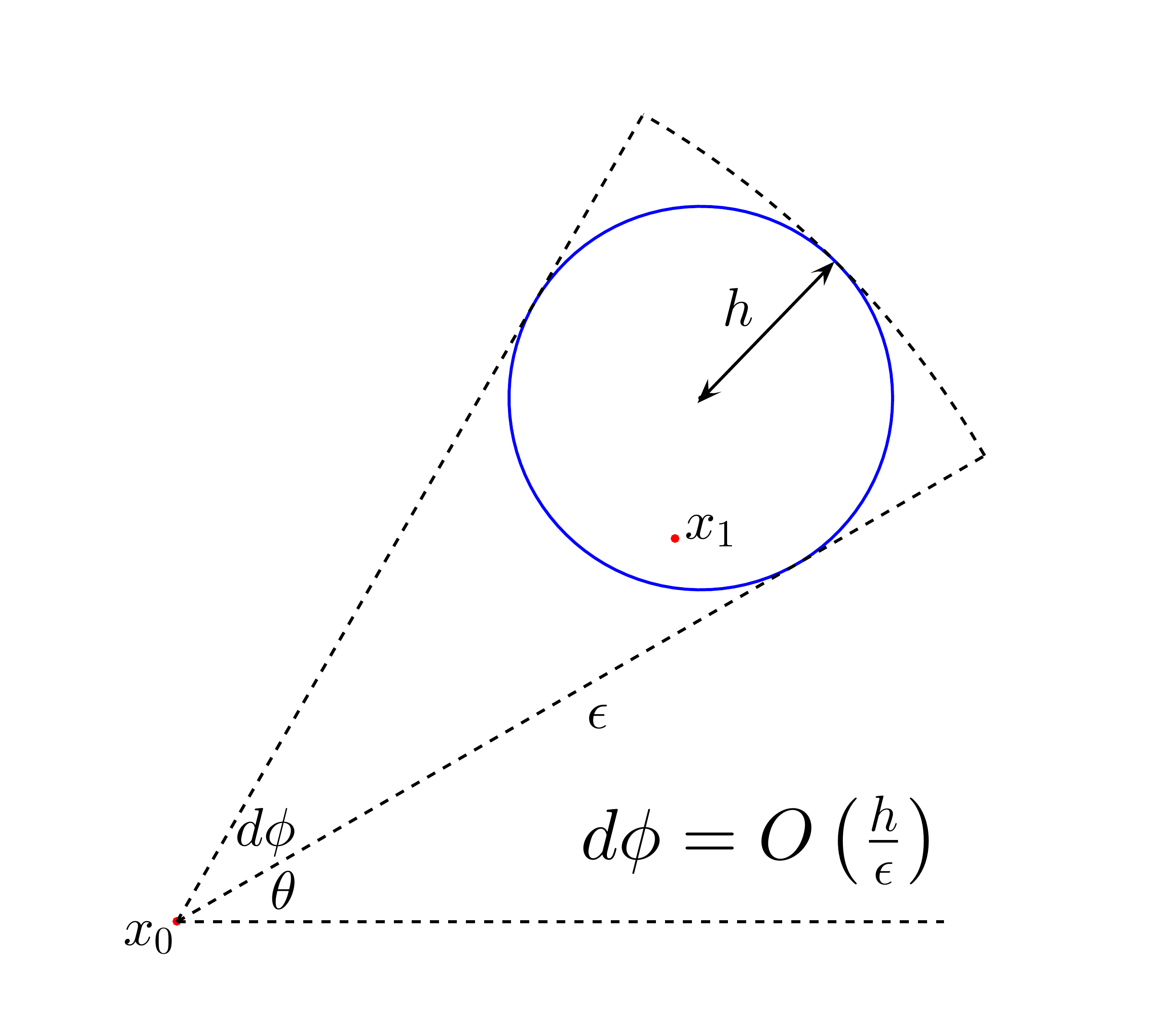}}\label{fig:MeshfreeExistence}
\subfigure[]{\includegraphics[width=0.45\textwidth]{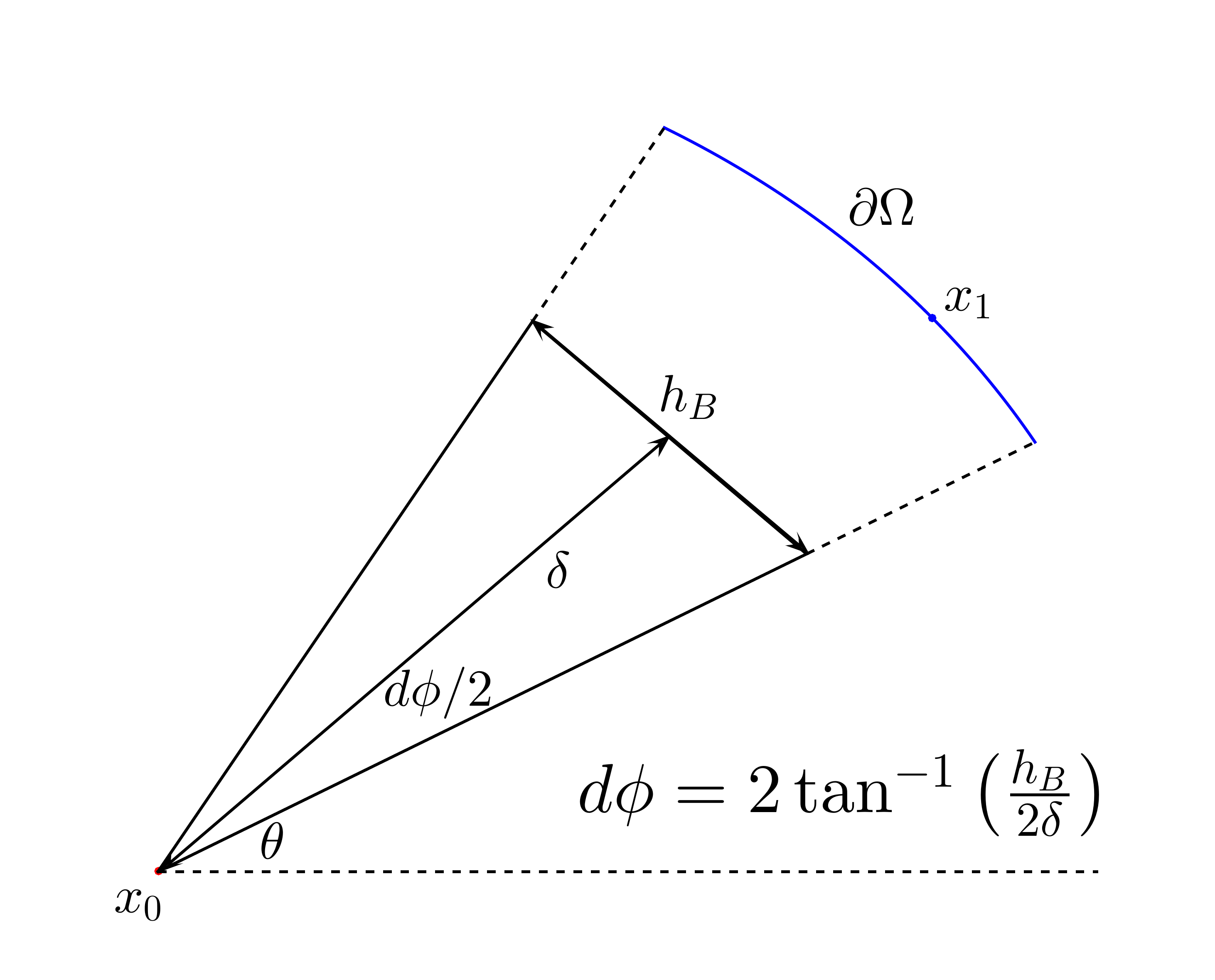}}\label{fig:MeshfreeBoundaryExist}
\caption{The angular resolution of a generalized finite difference stencil.}
\label{fig:MeshfreeAngular}
\end{figure}

Once any second directional derivative can be approximated, it is easy to substitute these into nonlinear operators of the form~\eqref{eq:PDE1}.  Functions of the eigenvalues of the Hessian~\eqref{eq:PDE2} are also easily approximated in two dimensions via the Rayleigh quotient characterization,
\bq\label{eq:eigs2D} \lambda_1(D^2u) = \min\limits_{\abs{\bm{\nu}}=1}u_{\bm{\nu}\bm{\nu}}, \quad \lambda_2(D^2u) = \max\limits_{\abs{\bm{\nu}}=1}u_{\bm{\nu}\bm{\nu}}. \eq

%\section{Convergence Analysis}\label{sec:analysis}
%\input{PaperAnalysis.tex}

\section{Discretization}\label{sec:numerics}
\subsection{Construction of the grid}\label{sec:grid}
	In order to build a grid or point cloud on which monotone schemes can be constructed efficiently for general domains, we seek to extend the framework used in~\cite{FroeseMeshfreeEigs} and summarized in \autoref{sec:genFD}. 
	Structured grids provide certain advantages in building the stencils quickly, but as in two dimensions, the boundary needs to be more resolved than the interior and it is necessary to preserve an appropriate gap $\delta$ between interior and boundary points. 
	In two dimensions, where the boundary is a one-dimensional curve, this is fairly straightforward. 
	However, it is much more difficult to find an optimal sampling of boundary points in three dimensions.
	
	We begin by identifying interior points. 
	Our strategy is to begin with a uniform discretization of a cube $C$ covering the domain ($\Omega\subset C$), then reduce to only the interior points. 
	Denote the grid by $\G$. Define $\xvec_{ijk}$, $i,j,k = 0, \ldots, n$ to be the nodes of the discretized cube $C$ and let $h$ be the space between adjacent nodes.

Next we define the signed-distance function to the boundary of the domain $\Omega$, 
	$$G(\xvec) = \begin{cases}
	\dist(\xvec,\partial X) & \xvec \notin X \\
	-\dist(\xvec,\partial X) & \xvec \in X \\
	0 & \xvec \in \partial X.
	\end{cases}$$
	As in two dimensions, we will require that there be some separation $\delta$ between the interior and the boundary in order to consistently resolve directional derivatives near the boundary. 
	Thus, the interior points in $\G$ are chosen to be
	$$\xvec_{ijk} \in C \text{ s.t. } G(\xvec_{ijk}) + \delta<0.$$
	 This ensures that there is a distance of at least $\delta$ between the boundary and any interior point. A two-dimensional visualization of this process is shown in Figure~\ref{fig:domain}.
	 
	Note that although we start with the discretization of a cube, this restriction can be applied to arbitrarily complicated three-dimensional regions.

\begin{figure}[htp]
\centering
\subfigure[]{
\includegraphics[width=0.45\textwidth]{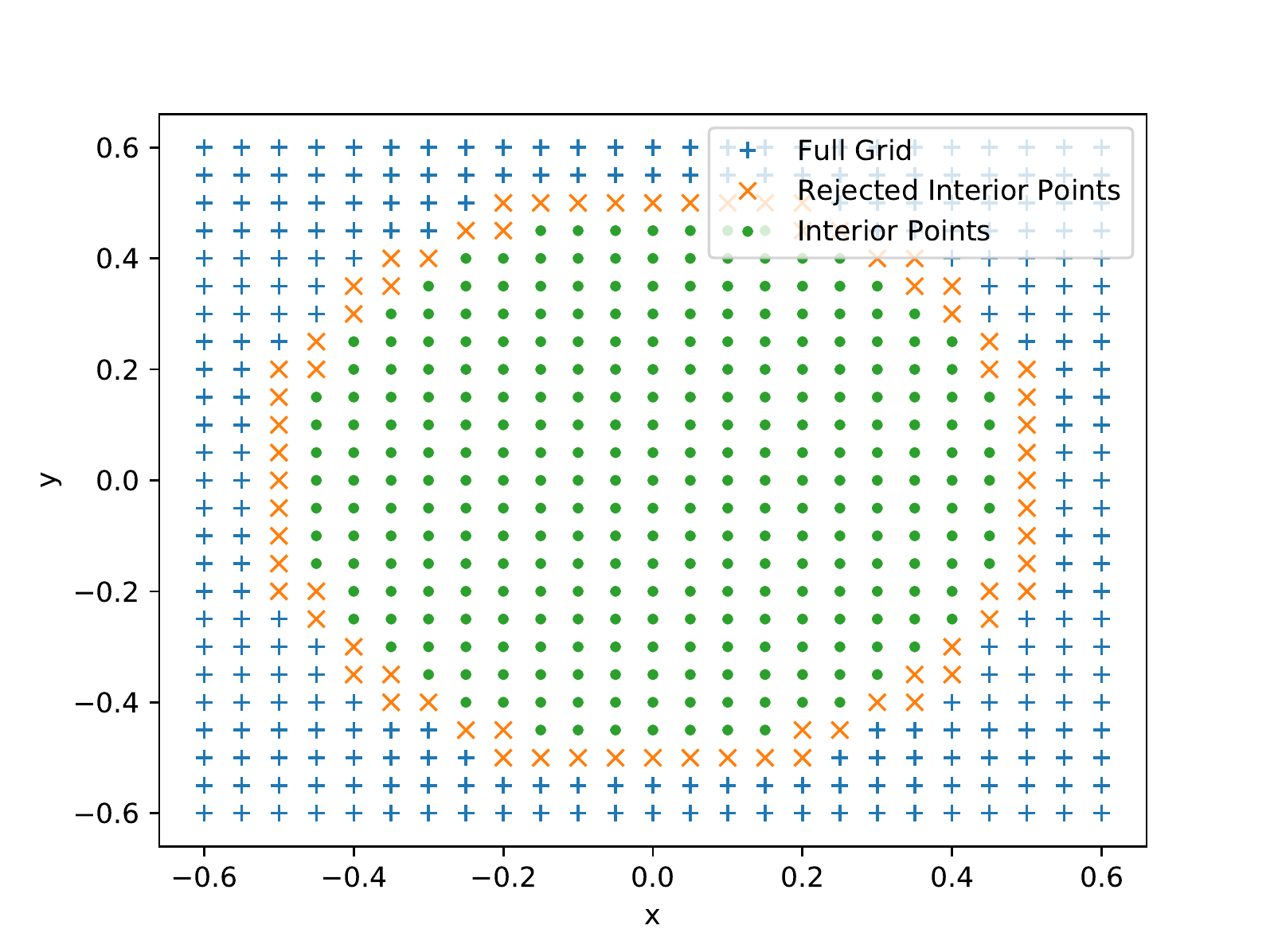}\label{fig:domain}}
\subfigure[]{
\includegraphics[width=0.45\textwidth]{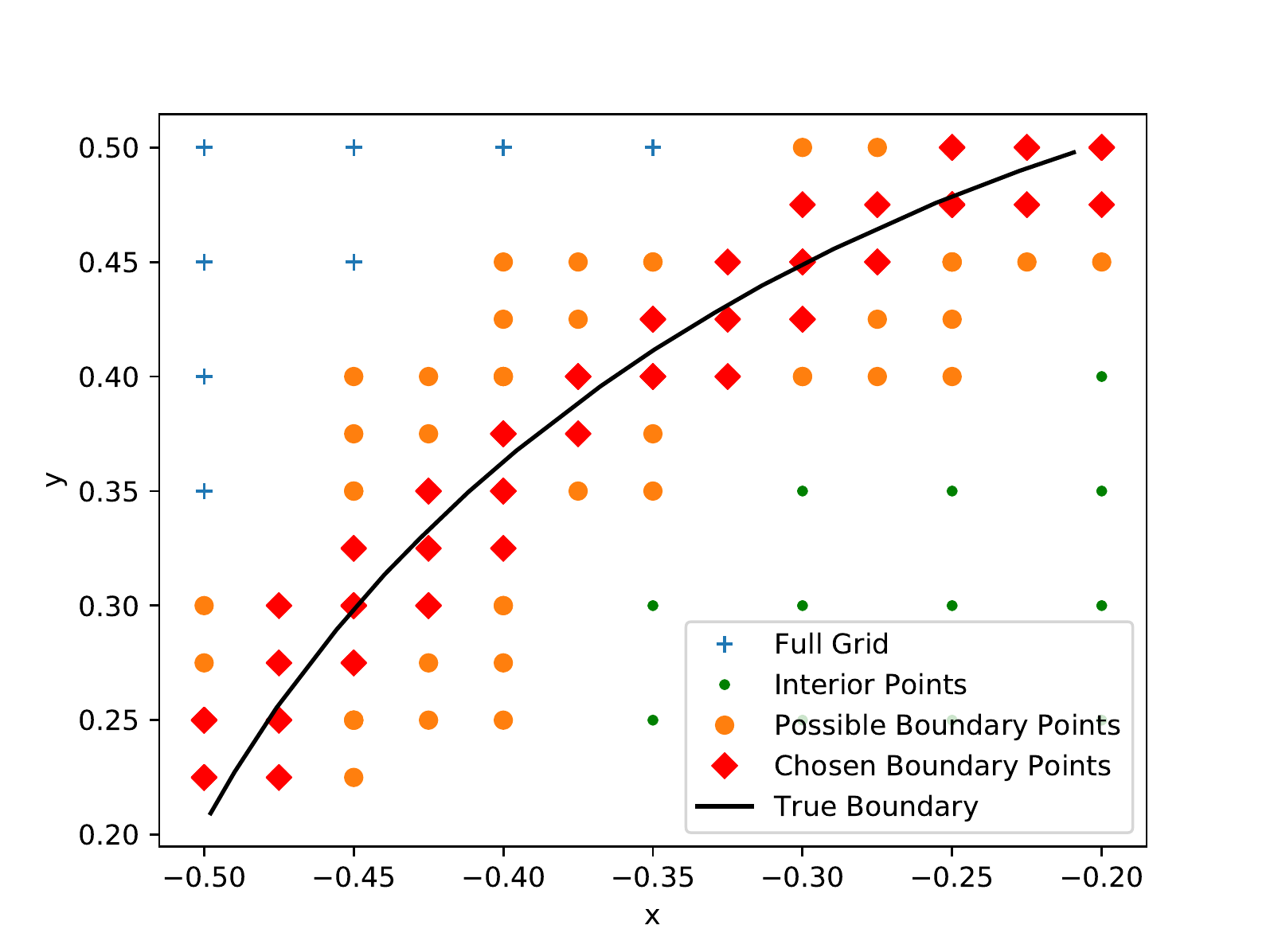}\label{fig:boundary}}
\caption{A two-dimensional visualization of the construction of the grid. \subref{fig:domain}~Candidate interior points, with points too close to the boundary rejected. \subref{fig:boundary}~Candidate boundary points, with selected boundary points projected onto the true boundary.}
\label{fig:domainConstruction}
\end{figure}

	%\begin{figure}[htbp]
	%\begin{center}
	%\includegraphics[width=\figwidth]{DomainConstruction}
	%\end{center}
	%\caption{Only the points at least $\frac{h}{2}$ from the boundary are kept in the point cloud.}
	%\label{3dinterior}
	%\end{figure}
	
	Next we describe the discretization of the boundary, which must have an effective resolution $h_B \ll \delta \leq \bO(h)$ in order to produce consistent schemes near the boundary. 
	
At each candidate interior point near the boundary, we focus on the following small cubes  
	\bq\label{eq:smallCube} C_{ijk} =[x_i,x_{i+1}]\times[y_j,y_{j+1}]\times[z_k,z_{k+1}],\eq
which we define for any $i,j,k$ such that $C_{ijk} \cap \partial \Omega$ is non-empty.  To identify these, we seek any such cube such that at least one corner $\xvec_{-} \in C_{ijk}$ satisfies $G(\xvec_{-})<0$ and one corner $\xvec_{+} \in C_{ijk}$ satisfies $G(\xvec_{+})>0$. 

	Then boundary points are added to the point cloud by further discretizing these boundary cubes and using the projection of points sufficiently close to the boundary of the domain. 
	Let $C_{ijk}$ be a boundary cube. 
	We introduce the discretization 
	\bq\label{eq:bdyCubes}D_{ijk} = \left\{(x_i+\tilde{i}h_B,y_j+\tilde{j}h_B,z_k+\tilde{k}h_B)\text{ s.t. } 0 \leq \tilde{i},\tilde{j},\tilde{k} \leq n_B\right\}\eq
where the boundary resolution $h_B = \bO(h/n_B)$.
	
As candidates for boundary points, we select any points $\xvec\in D_{ijk}$ such that  $G(\xvec) < \frac{h_B}{2}$.  We then project each of these candidates onto the true boundary and include the results $\text{Proj}_{\partial\Omega}(\xvec)$ in our point cloud $\G$.
	A two-dimensional visualization of this process is given in Figure~\ref{fig:boundary}.
	
	We emphasize again that Hypothesis~\ref{hyp:grid} requires this procedure to lead to an over-resolution of the boundary ($h_B\ll h$), which is necessary (as in 2D) in order to preserve both consistency and monotonicity up to the boundary.  There are $\Or(n^2)$ boundary cubes, each of which contains $\bO(n_B^3)$ points.  Of these, we select $\bO(n_B^2)$  to project onto the boundary. This leaves us with a total of $\Or(n^2n_B^2)$ boundary points in $\G$, as compared with $\bO(n^2)$ boundary points in a traditional three-dimensional Cartesian grid.

\subsection{Approximation of second derivatives}
Next we describe a process for constructing a (negative) monotone approximation of the second directional derivative $u_{\bm{\nu}\bm{\nu}}(\xvec_0)$.  	We look for schemes of the form 
	\begin{align}\label{eq:DirDeriv} u_{\bm{\nu}\bm{\nu}}(\xvec_0) \approx \Dt_{\bm{\nu}\bm{\nu}}u_0 = \sum_{j=1}^m a_j(u(\xvec_j)-u(\xvec_0))\end{align} 
	where each $\xvec_j \in \G \cap B(\xvec_0,\epsilon)$ is a nearby grid point and all $a_j \geq 0$ for monotonicity.

In the simplest setting, where the direction $\bm{\nu}$ is grid aligned (Figure~\ref{fig:aligned}), we can simply use a standard centered difference discretization.  That is, suppose that $\bm{\nu}\in\Zf^3$ with $\abs{\bm{\nu}}h < \epsilon$ and $\xvec_0 \pm \bm{\nu} h \in \G$.  Then we define
\bq\label{eq:d2center} \Dt_{\bm{\nu}\bm{\nu}}u_0 = \frac{u(\xvec_0+\bm{\nu} h) + u(\xvec_0 - \bm{\nu} h)-2u(\xvec_0)}{\abs{\bm{\nu}}^2h^2}, \eq
which satisfies all the requirements outlined above.

This is the approach taken by traditional wide stencil schemes.  However, there are some clear situations where this simple scheme is not available: (1) if $\bm{\nu}$ is not grid-aligned, (2) if $\bm{\nu}$ is grid-aligned but requires a stencil wider than our chosen search radius $\epsilon$, and (3) at points $\xvec_0$ near the boundary where one or both of $\xvec_0 \pm \bm{\nu} h$ can lie outside the domain.

	\begin{figure}[htbp]
	\subfigure{
	\includegraphics[trim=2cm 1.6cm 2cm 2.4cm,clip=true,width=2.25in]{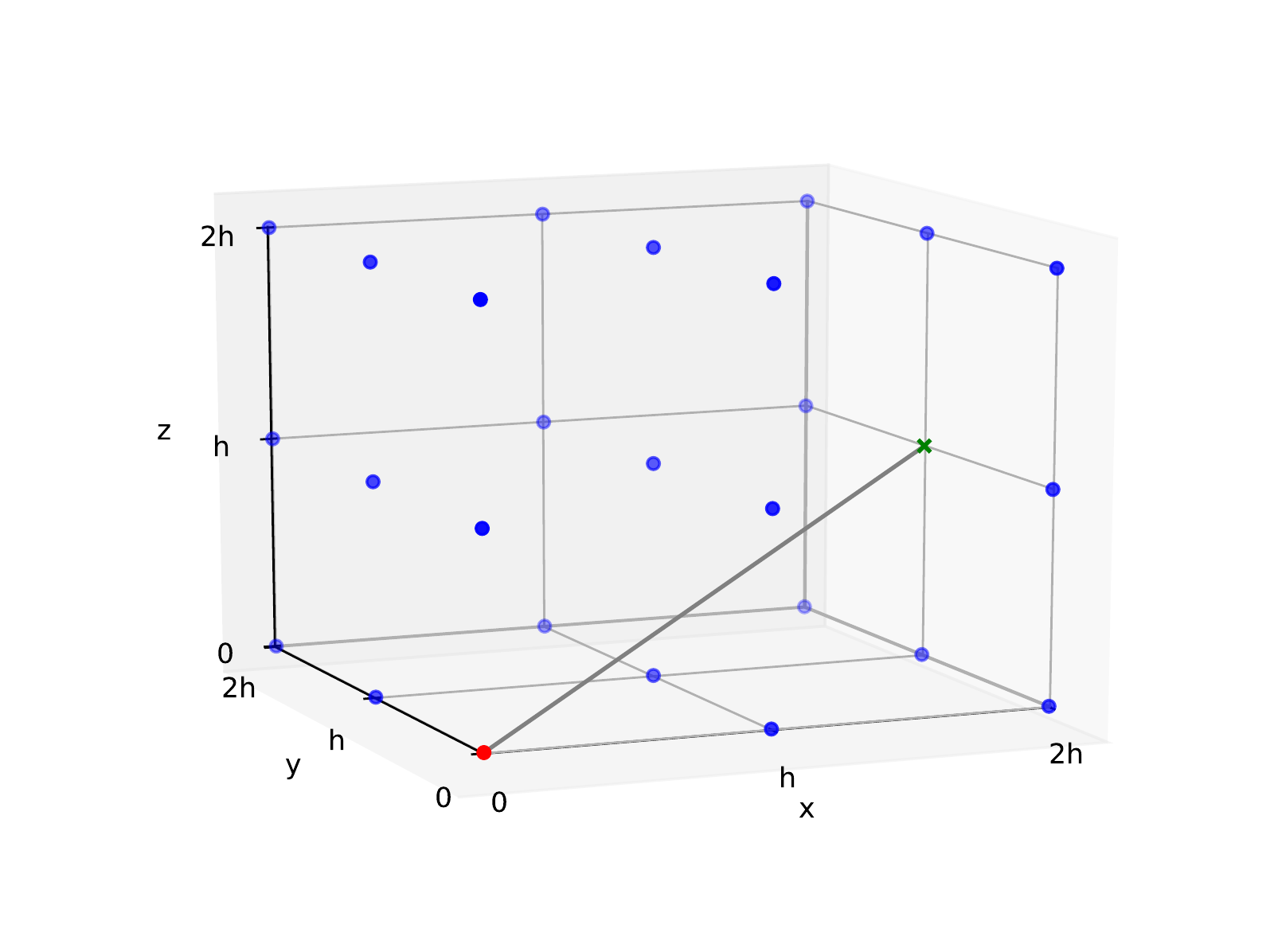}\label{fig:aligned}}\subfigure{\includegraphics[trim=2cm 1.6cm 2cm 2.4cm,clip=true,width=2.25in]{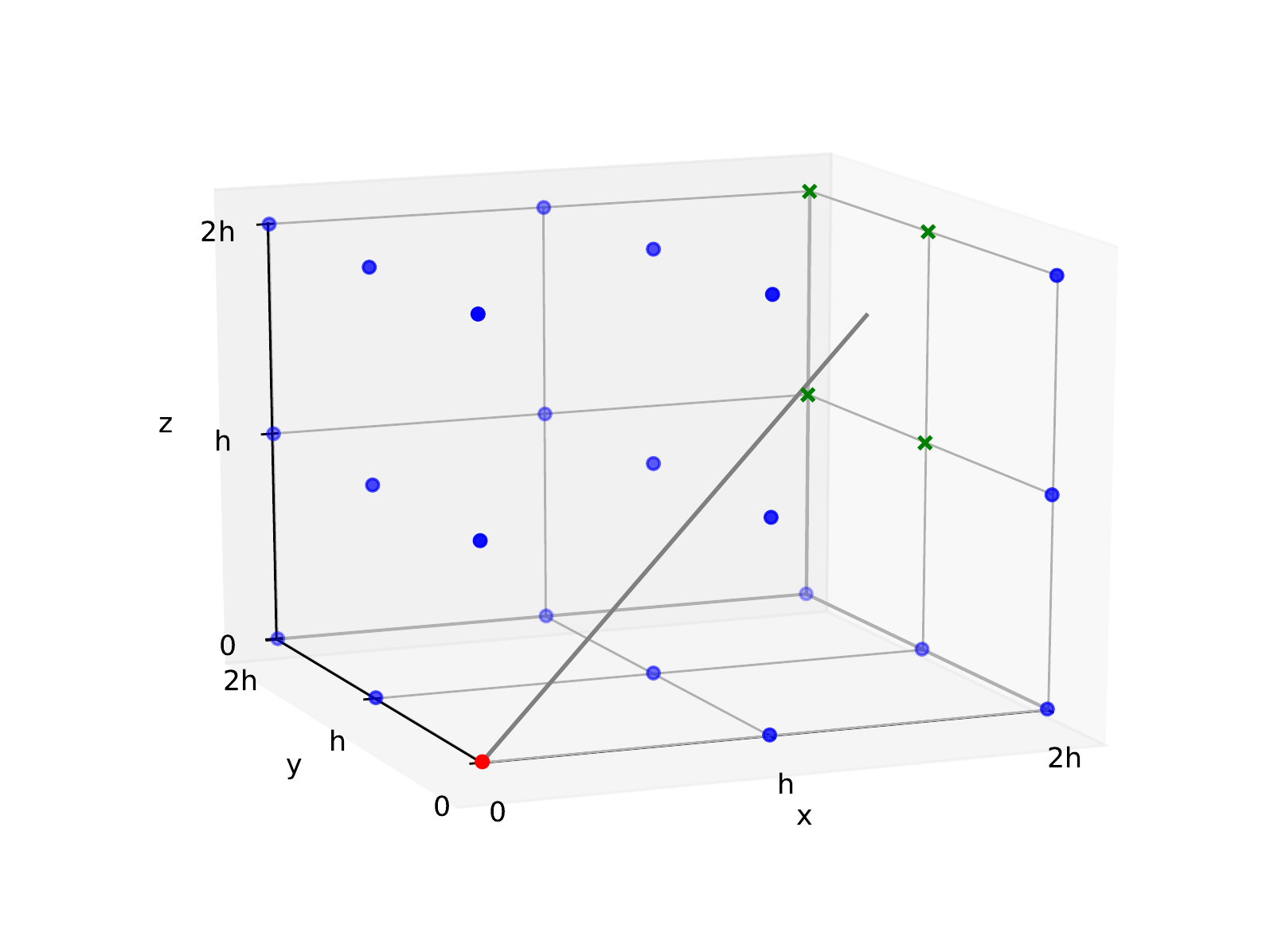}\label{fig:nonaligned}}
	%trim=left bottom right top
	\caption{Examples of \subref{fig:aligned}~a perfectly aligned neighbor along the vector $\xvec_0 + \bm{\nu} h$ and \subref{fig:nonaligned}~four non-aligned neighbors along this direction.}
	\label{fig:neighbors3D}
	\end{figure}

\subsubsection{Generalized finite difference schemes}\label{sec:d2gen}
At points where the simple centered scheme is not available, we will follow the approach of~\cite{FroeseMeshfreeEigs} and develop generalized finite difference schemes.  For clarity of exposition, we begin by considering the approximation of $u_{xx}(\xvec_0)$.  

From the reference point $\xvec_0$ (which is treated as the origin), we define an orthogonal coordinate system using the standard Cartesian coordinate axes $\hat{x}$, $\hat{y}$, $\hat{z}$.  Let $O_i$, $i=1, \ldots, 8$ denote the eight octants defined by these axes.  We can also use this coordinate frame to define spherical coordinates $(r_j,\theta_j,\phi_j)$ corresponding to any grid point $\xvec_j \in \G$.

In order to derive an appropriate scheme for $u_{xx}$ we rely on Taylor expansion.  That is, we seek a scheme of the form
\bq\label{eq:taylor}
\begin{aligned}
\Dt_{xx}u(\xvec_0) &= \sum\limits_{j=1}^m a_j(u(\xvec_j)-u(\xvec_0))\\
  &= \sum\limits_{j=1}^m a_j \left[u_x(\xvec_0)(x_j-x_0) + u_y(\xvec_0)(y_j-y_0) + u_z(\xvec_0)(z_j-z_0) \right.\\
		 &\phantom{=====}+ u_{xx}(\xvec_0)\frac{(x_j-x_0)^2}{2} + u_{yy}(\xvec_0)\frac{(y_j-y_0)^2}{2} + u_{zz}(\xvec_0)\frac{(z_j-z_0)^2}{2} \\
		 &\phantom{=====}+u_{xy}(\xvec_0)(x_j-x_0)(y_j-y_0) + u_{xz}(\xvec_0)(x_j-x_0)(z_j-z_0) \\
		 &\phantom{=====}+ \left. u_{yz}(\xvec_0)(y_j-y_0)(z_j-z_0)\right] + h.o.t.
\end{aligned}
\eq

To simplify this expansion, we want to choose neighbors $\xvec_j$ that are well-aligned with the $x$-axis so that $\abs{y-y_0}, \abs{z-z_0} = o(\abs{x-x_0})$.  We recall that in a perfectly grid-aligned scheme, we would have $y-y_0 = z-z_0 = 0$.  Inspired by the two-dimensional approach of~\cite{FroeseMeshfreeEigs}, we seek one point in each octant $O_j$ such that $\xvec_j-\xvec_0$ is as well-aligned as possible with the direction $\hat{x}$.  That is, we define
\bq\label{eq:nbs}
\xvec_j = \argmin\left\{\theta^2 + (\phi-\pi/2)^2 \mid \xvec \in O_j \cap \G \cap B(\xvec_0,\epsilon)\right\}.
\eq
See Figure~\ref{fig:nonaligned} for the selection of four neighbors well-aligned with the vector $\bm{\nu}$; the four neighbors along the direction $-\bm{\nu}$ are chosen similarly.

Notice that with slight modification, the spherical coordinates $\theta_j$ and $\phi_j-\pi/2$ of these neighbors now play the role of the angular resolution $d\phi$ introduced in \autoref{sec:genFD}.  In particular, we expect that $\theta_j, \phi_j-\pi/2 = \bO(h/\epsilon + h_B/\delta)$ as in two dimensions.  As desired, we obtain the relationship
\[ \abs{y_j-y_0}, \abs{z_j-z_0} \leq \bO(\epsilon\sin \theta_j \sin \phi_j + \epsilon\cos \phi_j) = \bO(\epsilon d\phi) \ll \bO(\epsilon) = \abs{x_j-x_0}. \]

Returning to the Taylor expansion in~\eqref{eq:taylor}, we now neglect all terms that are $o(\epsilon^2)$.  This leaves us with
\bq\label{eq:taylorSimple}
\begin{aligned}
\Dt_{xx}u(\xvec_0) 
  &= \sum\limits_{j=1}^8 a_j \left[u_x(\xvec_0)(x_j-x_0) + u_y(\xvec_0)(y_j-y_0)  \phantom{\frac{x^2}{2}} \right.\\
		 &\phantom{=====}\left.+ u_z(\xvec_0)(z_j-z_0)+ u_{xx}(\xvec_0)\frac{(x_j-x_0)^2}{2}\right] + h.o.t.  \\
\end{aligned}
\eq

Consistency of the approximation, combined with the monotonicity requirement, leaves us with the following system of equations for the coefficients $a_j$:
		\bq
		\begin{cases}
		\sum\limits_{j=1}^8 a_j (x_j-x_0) = 0\\
		\sum\limits_{j=1}^8 a_j (y_j-y_0) = 0\\
		\sum\limits_{j=1}^8 a_j (z_j-z_0) = 0\\
		\sum\limits_{j=1}^8 a_j \frac{(x_j-x_0)^2}{2}  = 1\\
		a_j \geq 0.
		\end{cases}
		\label{eq:FarkSyst}
		\eq
		
In two dimensions, an explicit solution of the system resulting from this procedure could be obtained.  This is not straightforward in three dimensions, and explicit formulas are often computationally intractable in practice.  Instead, we solve the system~\eqref{eq:FarkSyst} numerically through a least squares procedure.  That is, we notice that this system has the form
\[
\begin{cases}
M\mathbf{a} = \mathbf{b}\\
\mathbf{a} \geq 0
\end{cases}
\]
where $M$ is a $4\times 8$ matrix.  Thus, we can easily obtain a solution by solving the low-dimensional least squares problem
\bq\label{eq:leastSquares}
\begin{cases}
\text{minimize } &\frac{1}{2} \|M\mathbf{a}-\mathbf{b}\|_2^2 \\
\text{subject to } &\mathbf{a} \geq 0.
\end{cases}
\eq

We remark that this procedure is easily adapted to the approximation of more general second directional derivatives $u_{\bm{\nu}\bm{\nu}}(\xvec_0)$. To accomplish this, we introduce any two vectors $\bm{\nu}_2, \bm{\nu}_3$ such that $\bm{\nu}, \bm{\nu}_2, \bm{\nu}_3$ form an orthonormal set.  We use these as our coordinate axes, centered at the point $\xvec_0$, and introduce the change of coordinates
		$$
		\begin{cases}
		\bar{x}_j = (\xvec_j-\xvec_0)\cdot \bm{\nu},\\
		\bar{y}_j = (\xvec_j-\xvec_0)\cdot \bm{\nu}_2,\\
		\bar{z}_j = (\xvec_j-\xvec_0)\cdot \bm{\nu}_3.
		\end{cases}
		$$
The coefficients $a_j$ in the approximation $\Dt_{\bm{\nu}\bm{\nu}}(\xvec_0)$ in~\eqref{eq:DirDeriv} are then obtained by solving the system~\eqref{eq:FarkSyst} using these new coordinates in place of $x_j, y_j, z_j$.

	\subsubsection{Existence of a Positive Solution}
Our procedure for generating consistent, monotone generalized finite difference schemes in 3D hinges on finding a solution of~\eqref{eq:FarkSyst} via a least squares procedure.  However, it is by no means obvious that a solution satisfying the positivity requirement ($a_j \geq 0$) actually exists in general. Fortunately, this is guaranteed by our careful choice of neighboring points $\xvec_j$ lying in different octants.

The proof of this relies on Farkas' Lemma~\cite{FarkasLecture}.
		\begin{lemma}[Farkas' Lemma]\label{lem:farkas}
		Let $M \in \R^{m\x n}$ and $\mathbf{b} \in \R^{m \x 1}$. Then exactly one of the following two conditions holds:\\
		\begin{itemize}
		\item There exists $ \mathbf{a} \in \R^{n \x 1}$ such that $M\mathbf{a}=\mathbf{b}$ and $\mathbf{a} \geq 0$; 
		\item There exists $ \mathbf{y} \in \R^{m \x 1}$ such that $M^T \mathbf{y} \geq 0,  \mathbf{y}^T \mathbf{b} < 0$.
		\end{itemize}
		\end{lemma}

This allows us to prove the existence of a solution to the scheme~\eqref{eq:FarkSyst}, which immediately yields existence of a consistent and monotone scheme for $u_{\bm{\nu}\bm{\nu}}(\xvec_0)$.
		\begin{lemma}[Existence of positive solution]\label{lem:positiveSolution}
		A positive solution to the system of equations~\eqref{eq:FarkSyst} exists if the eight neighbors $\xvec_j$ lie in different octants as required by~\eqref{eq:nbs}. 
		\end{lemma}
		\begin{proof}
		Since each of the points $\xvec_j$ is chosen to lie in the octant $O_j$, we can assign a definite sign to each of the $x_j-x_0$, $y_j-y_0$, and $z_j-z_0$.  The system~\eqref{eq:FarkSyst} then takes the form
		\[
		\begin{cases}
		M\mathbf{a} = \mathbf{b}\\
		\mathbf{a} \geq 0
		\end{cases}
		\]
		where
		\begin{align}M = \begin{bmatrix}
		c_{11} & -c_{12} & -c_{13} &c_{14} &c_{15} &-c_{16} &-c_{17} &c_{18} \\
		c_{21} & c_{22} & -c_{23} &-c_{24} &c_{25} &c_{26} &-c_{27} &-c_{28} \\
		c_{31} & c_{32} & c_{33} &c_{34} &-c_{35} &-c_{36} &-c_{37} &-c_{38} \\
		c_{41} & c_{42} & c_{43} &c_{44} &c_{45} &c_{46} &c_{47} &c_{48}
		\end{bmatrix}, \quad \mathbf{b} = \begin{bmatrix} 0\\0\\0\\1\end{bmatrix},
		\end{align}
		and all of the $c_{ij} \geq 0$.
		
		Following Farkas' Lemma, suppose there exists some $\mathbf{y}$ such that $A^T \mathbf{y} \geq 0$ and $\mathbf{b}^T \mathbf{y} < 0$.  Notice that this second condition implies that
		$y_4 < 0.$
There are eight possible combinations of signs for the remaining components  $y_1, y_2,$ and $y_3$.

Consider, for example, the case where $y_1, y_2, y_3 \geq 0$.  Then we would have
\[ (M^T\mathbf{y})_7 = -c_{17}y_1-c_{27}y_2-c_{37}y_3+c_{47}y_4 < 0. \]
Similarly, we can verify that any other possible combination of signs in $y_1, y_2, y_3$ will require at least one component of $M^T\mathbf{y}$ to be negative. 

We conclude that there is no $\mathbf{y}$ such that both $M^T\mathbf{y} \geq 0$ and $\mathbf{y}^T\mathbf{b} < 0$.  By Farkas' Lemma, we infer the existence of a solution to~\eqref{eq:FarkSyst}.
		\end{proof}
		
We should also remark that this same strategy can be used using one perfectly aligned neighbor along the direction $\bm{\nu}$ (the setting of Figure~\ref{fig:aligned}) and four non-aligned neighbors along the direction $ - \bm{\nu}$ (the setting of Figure~\ref{fig:nonaligned}).  This situation can easily arise near the boundary, involving one perfectly aligned interior neighbor and four non-aligned boundary neighbors.  The proof is identical to that of the lemma above.
		
\begin{cor}[Existence of a consistent, monotone scheme]\label{cor:schemeExists}
Under the assumptions of Hypothesis~\ref{hyp:grid}, let $\bm{\nu}$ be any unit vector in $\R^3$.  Then the procedure described in \autoref{sec:d2gen} yields a consistent, monotone approximation of the negated second directional derivative $-u_{\bm{\nu}\bm{\nu}}(\xvec_0)$.
\end{cor}

\subsection{Approximation of nonlinear operators}\label{sec:Hessianevals}
	We can utilize these monotone approximations of the second directional derivatives to solve a wide class of fully nonlinear elliptic PDEs.  In this case of equations of the type~\eqref{eq:PDE1} that explicitly depend on directional derivatives over a finite subset of directions, an appropriate discretization is immediate.
\bq\label{eq:approx1}
F^h(\xvec,u(\xvec),u(\xvec)-u(\cdot)) = F\left(\xvec,u(\xvec),\Dt_{\bm{\nu}\bm{\nu}}u(\xvec); \bm{\nu}\in\Af\right), \quad \xvec \in \G^h\cap\Omega.
\eq
Monotonicity follows immediately from the fact that the elliptic operator $F$ is a non-increasing function of the second directional derivatives, which are approximated with a negative monotone discretization.
	
We are also interested in functions of the eigenvalues of the Hessian~\eqref{eq:PDE2}.  The Rayleigh quotient formulation that was used in two dimensions~\eqref{eq:eigs2D} does not immediately provide an expression for all three of the eigenvalues that we need to consider in three dimensions.  Straightforward modifications of this formula, involving minimization over an appropriate subspace of $\R^3$ to obtain the middle eigenvalue $\lambda_2(D^2u)$, do not lead to monotone approximations for $\lambda_2(D^2u)$.

Instead, we use properties from linear algebra to provide a new formulation that will allow for an appropriate monotone approximation of nonlinear functions of $\lambda_1(D^2u), \lambda_2(D^2u), \lambda_3(D^2u)$. 
To accomplish this, we first consider all sets of orthonormal coordinate frames in $\R^d$ (where we are particularly interested in $d=3$ in this work).	
	\bq\label{eq:orthCoords}V = \{ (\bm{\nu}_1,...,\bm{\nu}_d) \mid \bm{\nu}_j \in \R^d, \|\bm{\nu}_j\|_2 = 1,\bm{\nu}_j \perp \bm{\nu}_i \A i \neq j \}.\eq
	
	\begin{lemma}[Functions of eigenvalues]\label{lem:functionEigs}
	Let $G:\R\to\R$ be non-increasing, $\phi:\R\to\R$ concave, and $A$ a symmetric real-valued $d\times d$ matrix. Then
	\bq\label{eq:functionEigs}G\lp\sum_{j=1}^{d} \phi(\la_j(A))\rp = \max_{(\bm{\nu}_1,\bm{\nu}_2,...\bm{\nu}_d) \in V}G\lp \sum_{j=1}^d \phi(\bm{\nu}_j^TA\bm{\nu}_j)\rp.\eq
	\end{lemma}
	\begin{proof}
	Since $A$ is a real-valued symmetric matrix, we can find $d$ orthonormal eigenvectors $\mathbf{v}_1, \ldots, \mathbf{v}_d$. Any $(\bm{\nu}_1,\bm{\nu}_2,...\bm{\nu}_d) \in V$ can be expressed as a linear combination of these eigenvectors:
	$$\bm{\nu}_j = \sum_{k=1}^d c_{jk} \mathbf{v}_k = \sum_{k=1}^d (\bm{\nu}_j^T \mathbf{v}_k)\mathbf{v}_k.$$
	Since $\bm{\nu}_j$ and $\mathbf{v}_j$ are both orthonormal, we can also compute
	$$\sum_{k=1}^d c_{jk}^2 = \lp\sum_{k=1}^d c_{jk}\mathbf{v}_k\rp\lp\sum_{l=1}^d c_{jl} \mathbf{v}_l\rp = \bm{\nu}_j^T \bm{\nu}_j = 1,$$
	$$\sum_{k=1}^d c_{jk}^2 = \mathbf{v}_k^T \lp\sum_{j=1}^d \bm{\nu}_j\bm{\nu}_j^T\rp \mathbf{v}_k = \mathbf{v}_k^T I \mathbf{v}_k =\mathbf{v}_k^T \mathbf{v}_k= 1.$$

Now for any unit vector $\bm{\nu}_j$, we can use Jensen's inequality to estimate
\[ \phi(\bm{\nu}_j^TA\bm{\nu}_j) = \phi\left(\sum\limits_{k=1}^d c_{jk}^2\lambda_k\right) \geq \sum\limits_{k=1}^d c_{jk}^2\phi(\lambda_k).  \]
Summing these concave functions yields
\[ \sum\limits_{j=1}^d\phi(\bm{\nu}_j^TA\bm{\nu}_j) \geq \sum\limits_{j=1}^d\sum\limits_{k=1}^d c_{jk}^2\phi(\lambda_k)= \sum\limits_{k=1}^d\phi(\lambda_k) \]
with equality if the $(\bm{\nu}_1, \ldots, \bm{\nu}_d)$ coincide with the eigenvectors $(\mathbf{v}_1, \ldots, \mathbf{v}_d)$ of $A$.

Since $G$ is non-increasing, we conclude that
\[G\lp\sum_{j=1}^{d} \phi(\la_j(A))\rp = \max_{(\bm{\nu}_1,\bm{\nu}_2,...\bm{\nu}_d) \in V}G\lp \sum_{j=1}^d \phi(\bm{\nu}_j^TA\bm{\nu}_j)\rp. \qedhere\]
	\end{proof}
	\begin{remark}
	%This also applies if $G$ is non-decreasing and $\phi$ convex, with the maximum replaced with a minimum.
	This also applies if $G$ is non-decreasing and $\phi$ is convex, and if the maximum is replaced with a minimum, $G$ can be non-decreasing with $\phi$ concave, or non-increasing with $\phi$ convex.
	\end{remark}
	
This formulation immediately suggests a consistent, monotone approximation of the functions of the eigenvalues of the Hessian $D^2u$ since $\bm{\nu}^T (D^2u) \bm{\nu}$ is identical to the second directional derivative $u_{\bm{\nu}\bm{\nu}}$.  That is, for equations of the form~\eqref{eq:PDE2},
\bq\label{eq:approxfuneigs} F(\lambda_1(D^2u),\lambda_2(D^2u),\lambda_3(D^2u)) \approx  \max_{(\bm{\nu}_1,\bm{\nu}_2,...\bm{\nu}_d) \in V}G\lp \sum_{j=1}^d \phi(\Dt_{\bm{\nu}_j\bm{\nu}_j}u)\rp. \eq
However, this is not computationally feasible as it requires computing a maximum over an infinite set of orthogonal coordinate frames.

	Instead, we must consider some finite subset $V^h$ of the possible orthogonal frames $V$.    
	We begin with a finite subset $E^h\subset\mathbb{S}^2$ of unit vectors in $\R^3$.  Then we let
	\[ V^h = \{(\bm{\nu}_1, \ldots, \bm{\nu}_d) \in V \mid \bm{\nu}_j\perp\bm{\nu}_i \forall j \neq i\}. \]
	We can define the angular resolution of this subset to be 
	\bq\label{eq:angularFrames}d\theta = \max\limits_{\mathbf{v}_1, ..., \mathbf{v}_d \in V}\min\limits_{(\bm{\nu}_1, ..., \bm{\nu}_d) \in V^h}\max\limits_i \cos^{-1}(\mathbf{v}_i\cdot{\bm{\nu}}_i).\eq
	That is, for each frame in $V$, we first find the frame in $V^h$ that minimizes the worst case angle between $\mathbf{v}_i$ and $\bm{\nu}_i$. 
	Then, we maximize over all possible frames in $V$ to find the worst case $d\theta$.

Many suitable choices of $V^h$ are possible, and this immediately leads to an appropriate discretization.
\begin{lemma}[Consistent, monotone approximation]\label{lem:scheme2Exists}
Consider a grid $\G$ satisfying Hypothesis~\ref{hyp:grid} and a finite set $V^h \subset V$ chosen so that the angular resolution $d\theta\to0$ as $h\to0$.  Let $G$ be continuous and non-increasing and $\phi$ be concave.  Then
\bq\label{eq:approx2}
F^h(\xvec,u(\xvec)-u(\cdot)) = \max_{(\bm{\nu}_1,\bm{\nu}_2,...\bm{\nu}_d) \in V^h}G\lp \sum_{j=1}^d \phi(\Dt_{\bm{\nu}_j\bm{\nu}_j}u(\xvec))\rp, \quad \xvec \in \G \cap \Omega
\eq
is a consistent, monotone approximation of~\eqref{eq:PDE2}.
\end{lemma}

\subsection{Boundary conditions}\label{sec:boundarycond}
We now turn our attention to the approximation of boundary conditions.  Dirichlet boundary conditions are straightforward.  However, we are also interested in constructing monotone schemes for Neumann or Robin boundary conditions, as well as the nonlinear second type (optimal transport) boundary condition $\nabla u(\Omega_1) \subset \bar{\Omega}_2$.
	
	\subsubsection{Approximation of first derivatives}\label{sec:derivs1}
We begin by describing the approximation of first directional derivatives in directions $\mathbf{n}$ exterior to the domain.  That is, letting $\mathbf{n}_\xvec$ be the unit outward normal to the domain at the point $\xvec\in\partial\Omega$, we discretize derivatives $u_\mathbf{n}(\xvec_0) = \nabla u(\xvec_0) \cdot \mathbf{n}$  for directions $\mathbf{n}$ satisfying $\mathbf{n} \cdot \mathbf{n}_{\xvec_0} > 0$.

As with the interior, simple schemes can be used if the direction $\mathbf{n}$ is well-aligned with the grid (locally at the boundary point $\xvec_0$).
That is, if $\mathbf{n}\in\Z^3$ and $\abs{\mathbf{n}}h < \epsilon$, we could utilize the upwind scheme
\bq\label{eq:scheme1aligned}
\Dt_\mathbf{n}(\xvec_0) = \frac{u(\xvec_0)-u(\xvec_0-\mathbf{n} h)}{\abs{\mathbf{n}}h}.
\eq
However, given that the boundary is highly resolved relative to the interior and that complicated domains are possible, we do not expect this simple approximation to be possible in general.
	
		In the interior, we were able to construct monotone schemes by choosing neighbors in different octants relative to the direction and the point of interest.
		On the boundary, a similar approach yields monotone schemes for the first directional derivatives using only interior neighbors.

Taylor expanding as before, we get
\bq\label{eq:taylorBdy}
\begin{aligned}
\Dt_{\mathbf{n}}u(\xvec_0) &= \sum\limits_{j=1}^m a_j(u(\xvec_j)-u(\xvec_0))\\
  &= \sum\limits_{j=1}^m a_j \left[u_x(\xvec_0)(x_j-x_0) + u_y(\xvec_0)(y_j-y_0) + u_z(\xvec_0)(z_j-z_0) \right] + h.o.t.
\end{aligned}
\eq

Consistency is achieved by equating the coefficients of the various first partial derivatives to the components $n_1,n_2,n_3$ of the unit direction $\hat{\mathbf{n}}$.  Coupled with the (positive) monotonicity requirement, we obtain the system
		\bq
		\begin{cases}
		\sum\limits_{j=1}^m a_j (x_j-x_0) = {n}_1\\
		\sum\limits_{j=1}^m a_j (y_j-y_0) = {n}_2\\
		\sum\limits_{j=1}^m a_j (z_j-z_0) = {n}_3\\
		a_j \leq 0.
		\end{cases}\label{eq:Nscheme}
		\eq

%As before, we can construct an orthogonal coordinate frame for which $n$ serves as one of the axes.  Now we select four neighbors, one in each of the four octants that are interior to the domain at the point $\xvec_0$.	That is, each of these four neighbors are in different octants and satisfy $(\xvec_j -\xvec_0) \cdot n \leq 0$.  Unlike in the discretization of second derivatives, we do \emph{not} require these neighbors to be well-aligned with the vector $-n$.
%
%There are many choices of neighbors satisfying this requirement.  However, we do not wish to use boundary points $\xvec_j \in \G \cap \partial\Omega$ since the boundary resolution $h_B$ is much less than the interior resolution $h$.  Thus mixing interior and boundary points in~\eqref{eq:Nscheme} is typically not optimal for accuracy and stability.

As a simple way of selecting appropriate neighbors, we let $C_{ijk}$ be the first small cube~\eqref{eq:smallCube} entered by the ray $\xvec_0 - t \mathbf{n}$.  We choose as neighbors $\xvec_1, \xvec_2, \xvec_3, \xvec_4$ the four vertices of the face through which this ray enters the small cube.
	
		%A monotone scheme can be constructed for the boundary by solving the system using these four neighbors. 
		%The existence of a solution to~\eqref{eq:Nscheme} can be proved via Farkas' Lemma similarly to the proof used at interior points. 
		
\begin{lemma}[Existence of a negative solution]\label{sec:positiveBdy}
A negative solution to the system of equations~\eqref{eq:Nscheme} exists if $\xvec_0 - t\mathbf{n}$ lies in the convex hull of the four vertices $\xvec_1, \xvec_2, \xvec_3, \xvec_4$  of a square for some $t>0$.
\end{lemma}

\begin{proof}
Since $\xvec_0-t\mathbf{n}$ lies in the convex hull of the four corners of a square, then it also lies in the convex hull of three of these points.  Without loss of generality, let these be $\xvec_1, \xvec_2, \xvec_3$.  Then there exist $\lambda_1, \lambda_2, \lambda_3 \in [0,1]$ with $\lambda_1+\lambda_2+\lambda_3 = 1$ such that
\[ \xvec_0-t\mathbf{n} = \lambda_1\xvec_1 + \lambda_2\xvec_2 + \lambda_3\xvec_3. \]

Now we let $v(\xvec)$ be the piecewise linear interpolant of the values of $u(\xvec)$ at the points $\xvec_0,\xvec_1, \xvec_2, \xvec_3\in\R^3$.  Since $v$ is linear, we can compute its first directional derivative in the direction $\mathbf{n}$ via
\begin{align*}
\frac{\partial v}{\partial \mathbf{n}} &= \frac{v(\xvec_0)-v(\xvec_0-t\mathbf{n})}{t}\\
  &= \frac{v(\xvec_0)-\lambda_1v(\xvec_1)-\lambda_2v(\xvec_2)-\lambda_3v(\xvec_3)}{t}.
\end{align*}

Then we can easily verify that
\[ a_1 = -\frac{\lambda_1}{t}, \, a_2 = -\frac{\lambda_2}{t}, \, a_3 = -\frac{\lambda_3}{t}, \, a_4 = 0 \]
is a solution of~\eqref{eq:Nscheme}.
\end{proof}

\begin{cor}[Existence of a consistent, monotone scheme]\label{cor:schemeExistsBdy}
Consider a grid $\G$ satisfying Hypothesis~\ref{hyp:grid} and let $\mathbf{n}$ be any vector in $\R^3$ exterior to the domain $\Omega$ at the point $\xvec_0\in\partial\Omega$.  Then the procedure described in \autoref{sec:derivs1} yields a consistent, monotone approximation of the first directional derivative $u_{\mathbf{n}}(\xvec_0)$.
\end{cor}

\subsubsection{Approximation of optimal transport conditions}\label{sec:BVP2}
	In optimal transport and many geometric PDEs~\cite{brendle2010}, a traditional boundary condition is replaced by the so-called second type boundary condition $\nabla u(\Omega_1) \subset \bar{\Omega}_2$ where $\Omega_2 \subset \R^3$ is convex and the solution $u$ is also required to be convex.
	
This global constraint can be re-expressed as a nonlinear Hamilton-Jacobi equation on the boundary
\bq\label{eq:HJ}  H(\nabla u(\xvec)) = 0, \quad \xvec \in \partial\Omega_1\eq
where $H$ is the signed distance function to the boundary of the target set $\Omega_2$.  By utilizing the Legendre-Fenchel transform, it is possible to rewrite this in the form
\bq\label{eq:HJBC}
\sup\limits_{\mathbf{n}\cdot \mathbf{n}_\xvec > 0}\left\{ \nabla u(\xvec) \cdot \mathbf{n} - H^*(\mathbf{n}) \right\} = 0, \quad \xvec \in \partial\Omega_1.
\eq
		
This immediately allows us to construct an appropriate discretization using our approximations for first directional derivatives $u_\mathbf{n}$ and the finite subset $E^h$ of unit vectors in $\R^3$.  
\begin{lemma}[Consistent, monotone approximation]\label{lem:scheme2Exists}
Consider a grid $\G$ satisfying Hypothesis~\ref{hyp:grid} and a finite set $E^h \subset \mathbb{S}^2$ chosen so that the angular resolution $d\theta\to0$ as $h\to0$.  Then
\bq\label{eq:approxHJBC}
H^h(\xvec,u(\xvec)-u(\cdot)) = \max\limits\left\{\Dt_{\bm{\nu}}(\xvec)-H^*(\mathbf{n}) \mid \mathbf{n} \in E^h, \, \mathbf{n}\cdot \mathbf{n}_\xvec > 0\right\}, \quad \xvec\in \G\cap\partial\Omega_1
\eq 
is a consistent, monotone approximation of~\eqref{eq:HJBC}.
\end{lemma}

\section{Implementation}\label{sec:implementation}
The preceding section shows how to define a consistent, monotone approximation for a wide range of fully nonlinear elliptic operators in three dimensions.  However, naive evaluation of these approximation schemes may be computationally intractable in three dimensions.  In this section, we discuss details of our implementation that allow us to evaluate these schemes efficiently.

\subsection{Discretization parameters}
There are many valid choices for the parameters used to construct our numerical discretization.  The particular choices used in our implementation are motivated by the need for efficiency that is brought to the forefront in three dimensions.

We begin with the parameters used to define the grid in \autoref{sec:grid}, particularly the boundary resolution $h_B$ and the gap to the boundary $\delta$.  We recall that the total number of boundary points scales like $\bO(n^2n_B^2)$ while the total number of interior points scales like $\bO(n^3)$.  While the boundary has to be more highly resolved than in a traditional finite difference grid, we still desire the number of boundary points to be less than the number of interior points to prevent this high resolution from significantly impacting computational cost.  With this in mind, we choose $n_B \approx n^{1/4}$, so that the total number of boundary points is $\bO(n^{5/2}) \ll n^3$.  Note that the boundary resolution is then $h_B = \bO(h/n_B) = \bO(h^{5/4})$, which is asymptotically less than $h$ as required by Hypothesis~\ref{hyp:grid}.  In order to satisfy the condition $h_B \ll \delta$, we choose $\delta = h/2$.

The procedure for approximating second directional derivatives also requires us to define a search radius $\epsilon\gg h$.  We recall that (as long as we are not too close to the boundary), the discretization error of these approximations depends on both the effective spatial resolution $\epsilon$ and the angular resolution $d\phi = \bO(h/\epsilon)$.  Motivated by the desire to balance these two components of the error, we choose $\epsilon = \sqrt{h}$.  We note that the discretization error near the boundary may be slightly larger (as it is also influenced by our choice of $h_B$ and $\delta$).  However, this occurs in only a narrow band near the boundary and need not necessarily affect the scaling of the overall error in the computed solution.

If we are considering functions of the eigenvalues of the Hessian~\eqref{eq:PDE2}, we also need to define a discretization $V^h$ of orthogonal coordinate frames.  In two dimensions, this is very straightforward.  However, in three dimensions, the number of possible coordinate frames can quickly become very large as the resolution $d\theta$ of $V^h$ is improved.  In light of the need to conserve computational resources, we would like to make use of the simple centered scheme~\eqref{eq:d2center} as much as possible.  For this reason, we restrict our attention to grid aligned directions remaining within our search radius and define
\bq\label{eq:Eh}
E^h = \left\{\bm{\nu}\in\Zf^3 \mid \|\bm{\nu}_j\|_\infty \leq \epsilon/h \right\}.
\eq
and
\bq\label{eq:Vh}
V^h = \left\{(\bm{\nu}_1,\bm{\nu}_2,\bm{\nu}_3) \mid \bm{\nu}_j \in E^h, \bm{\nu}_i\perp\bm{\nu}_j \A i \neq j\right\}.
\eq  
Then the more complicated generalized schemes only need to be constructed in a band of width $\epsilon$ around the boundary.  The resulting angular resolution will scale like $d\theta = \bO(h/\epsilon)$, which is comparable to the angular resolution $d\phi$ of the generalized finite difference stencils.

Since the second type boundary condition~\eqref{eq:HJBC} is typically coupled to PDEs that depend on the eigenvalues of the Hessian,
it is natural to use this same discretization of unit vectors in the approximation of this boundary condition.

\subsection{Identification of orthogonal coordinate frames}
	In this section, we discuss the evaluation of the expression in~\eqref{eq:approxfuneigs}, which requires computing a maximum/minimum over many different orthogonal frames in order to achieve a consistent approximation of functions of the eigenvalues of the Hessian.
	
Identifying all possible coordinate frames occurring in~\eqref{eq:Vh} can be done offline.  For each integer stencil width $k\in\Nf$, we can use brute force to construct and save the coordinate frames that can be constructed using this stencil width.  That is, we define
\bq\label{eq:Ek}
E_k = \left\{\bm{\nu}\in\Zf^3 \mid \|\bm{\nu}\|_\infty \leq k \right\}.
\eq
and
\bq\label{eq:Vh}
V_k = \left\{(\bm{\nu}_1,\bm{\nu}_2,\bm{\nu}_3) \mid \bm{\nu}_j \in E_k, \bm{\nu}_i\perp\bm{\nu}_j \A i \neq j\right\}.
\eq  
Then for a given problem (which may involve many different choices of domains, grid resolutions, particular PDEs, etc.), we simply define $V^h = V_{k^*}$ where $k^* = \floor{\epsilon/h}$.  As these have been pre-computed, there is effectively no computational cost to identifying the relevant coordinate frames.

However, actually solving a PDE involving the eigenvalues of the Hessian can require multiple evaluations of a minimum/maximum over all possible frames.  In three dimensions, this becomes very expensive.  For this reason, we propose a multi-level approach for obtaining the solution of these maximum/minimum problems.  We will focus the discussion on the problem of solving systems of the form
\bq\label{eq:maxProblem} \max\limits_{(\bm{\nu}_1,\bm{\nu}_2,\bm{\nu}_3)\in V_{k^*}} F(u;\bm{\nu}_1,\bm{\nu}_2,\bm{\nu}_3) = 0, \eq
which immediately provides a means for approximating schemes of the form~\eqref{eq:approxfuneigs}.

The idea of our approach is to first solve~\eqref{eq:maxProblem} over coordinate frames of maximum stencil width one.  From here, we identify the twenty-five coordinate frames (of maximum stencil width two) most closely aligned with the maximizer $(\bm{\nu}_1^{(1)},\bm{\nu}_2^{(1)},\bm{\nu}_3^{(1)})$ of the narrow stencil problem.  We once again solve~\eqref{eq:maxProblem} over this small set of possible coordinate frames.  This procedure can be repeated, maximizing over twenty-five coordinate frames at a time, until we are solving the system by maximizing over frames of the desired maximum width $k^*$. See Figure~\ref{fig:multlvl} for a visualization of this procedure.

	\begin{figure}[htbp]
	\begin{center}
	\includegraphics[width=\figwidth]{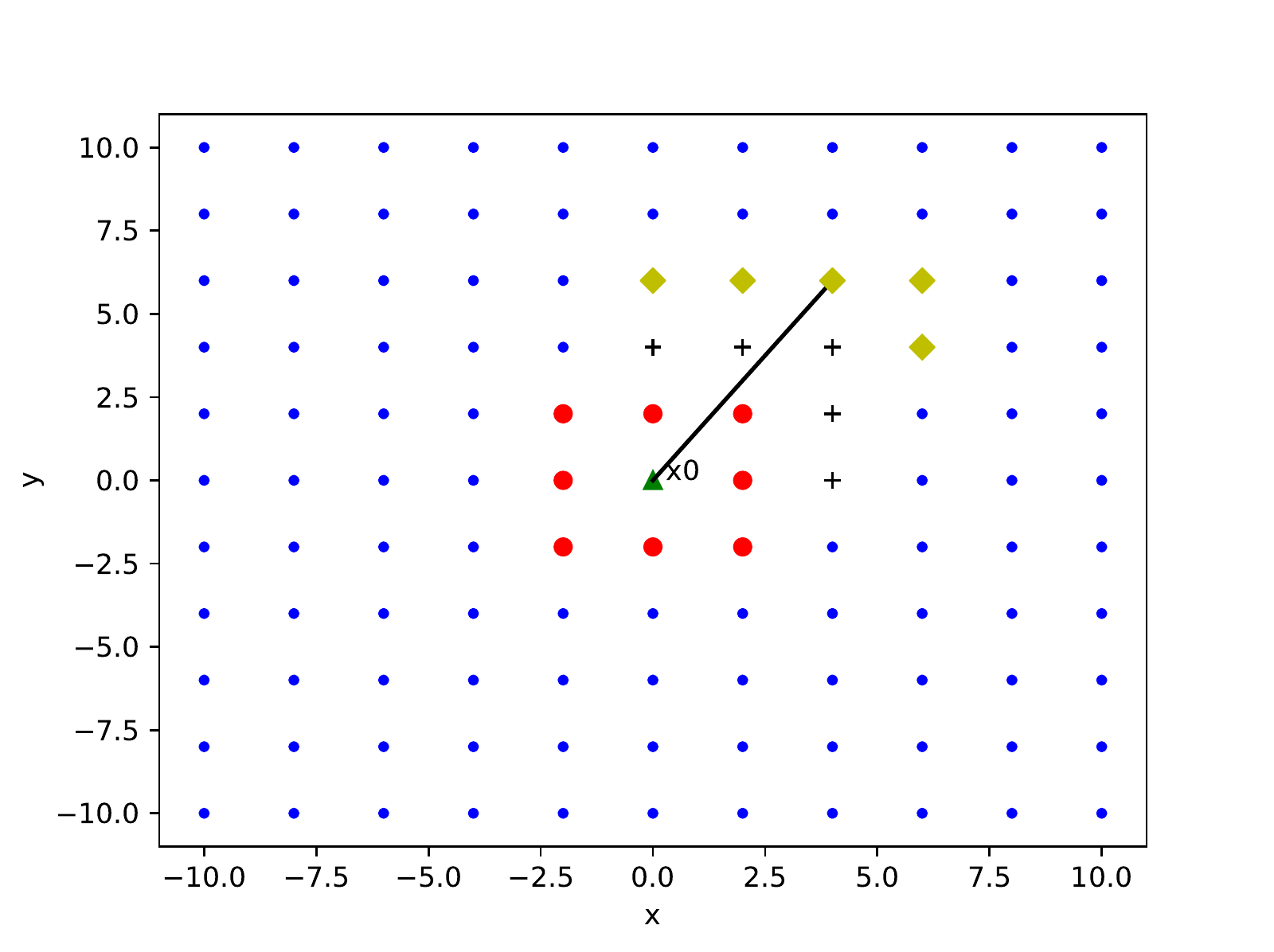}
	\end{center}
	\caption{A two-dimensional illustration of the multi-level process for one direction $\bm{\nu}_1$ in the orthogonal frame. The true direction that maximizes~\eqref{eq:maxProblem} is given by the black line. In the first level, we maximize over all the nearest (red dot) neighbors. We then identify the five (black plus) neighbors of stencil width two most closely aligned with the maximizer.  After maximizing over these five neighbors, we continue the procedure by identifying the best five (yellow diamond) neighbors of stencil width three.}
	\label{fig:multlvl}
	\end{figure}

We begin by producing a hierarchy of possible coordinate frames, which can be generated offline and stored.  To accomplish this, we explicitly identify the different components appearing in the sets of coordinate frames $V_k$. That is, we write
\[ V_k = \left\{(\bm{\nu}_1,\bm{\nu}_2,\bm{\nu}_3) \mid \bm{\nu}_1 \in V_k^{(1)}, \bm{\nu}_2\in V_k^{(2;\bm{\nu}_1)}, \bm{\nu}_3 = \bm{\nu}_1\times\bm{\nu}_2\right\} \]
where $V_k^{(1)}\subset E_k$ and 
\[ V_k^{(2;\bm{\nu}_1)} \subset \left\{\bm{\nu}_2\in E_k \mid \bm{\nu}_2\perp\bm{\nu}_1 \text{ for some }\bm{\nu}_1\in V_k^{(1)}\right\}. \]

Next, we create a map from each $\bm{\nu}_1\in V_k^{(1)}$ to the five most closely aligned vectors $\bm{\nu}\in V_{k+1}^{(1)}$.  We introduce the notation
\[ \alpha(\bm{\nu},\bm{\mu}) = \cos^{-1}\left(\frac{\bm{\nu}\cdot\bm{\mu}}{\|{\bm{\nu}}\|\|{\bm{\mu}}\|}\right) \]
to denote the angle between the vectors $\bm{\nu}$ and $\bm{\mu}$.  Then the map has the form
\[ T_{k+1}^{(1)}(\bm{\nu}_1) = \{\bm{\mu}_1,\bm{\mu}_2,\bm{\mu}_3,\bm{\mu}_4,\bm{\mu}_5\} \subset V_{k+1}^{(1)} \]
and satisfies the close alignment condition
\[ \alpha(\bm{\nu}_1,\bm{\mu}) \leq \alpha(\bm{\nu}_1,\bm{\xi}) \text{ for every } \bm{\mu}\in T_{k+1}^{(1)}(\nu_1), \bm{\xi}\in V_{k+1}^{(1)}-T_{k+1}^{(1)}(\bm{\nu}_1). \]

Similarly, we produce a map from each $\bm{\nu}_2 \in V_k^{(2;\nu_1)}$ to the five most closely aligned vectors in $V_{k+1}^{(2;\bm{\mu})}$.  This now has to be done for \emph{every} $\bm{\mu}\in V_{k+1}^{(1)}(\bm{\nu}_1)$ since we need to obtain nearby orthogonal coordinate frames, not merely nearby vectors.  That is, we define 
\[ T_{k+1}^{(2;\bm{\mu})}(\bm{\nu}_2) = \left\{\bm{\rho}_1,\bm{\rho}_2,\bm{\rho}_3,\bm{\rho}_4,\bm{\rho}_5\right\}\subset V_{k+1}^{(2;\bm{\mu})} \]
satisfying the close alignment condition
\[ \alpha(\bm{\nu}_2,\bm{\rho}) \leq \alpha(\bm{\nu}_2,\bm{\xi}) \text{ for every } \bm{\rho}\in T_{k+1}^{(2;\bm{\mu})}(\bm{\nu}_2), \bm{\xi}\in V_{k+1}^{(2;\bm{\mu})}-T_{k+1}^{(2;\bm{\mu})}(\bm{\nu}_2). \]

We emphasize again that all of the preceding work in building a hierarchy of maps can be accomplished offline and stored for later use in a wide variety of problems.  Then the actual work of solving~\eqref{eq:maxProblem} involves solving a small number of similar problems, each involving at most twenty-five possible coordinate frames.  The required online computations are summarized in the very short Algorithm~\ref{alg:maxProblem}.
	
\begin{algorithm}[h]
\caption{Estimating the solution of~\eqref{eq:maxProblem} over orthogonal coordinate frames.}
\label{alg:maxProblem}
\begin{algorithmic}[1]
\State $W_1 = V_1$
\For{$k=1,\ldots,k^*-1$}
\State $u$ $\Leftarrow$ solution of $\quad \max\limits_{(\bm{\nu}_1,\bm{\nu}_2,\bm{\nu}_3)\in W_k}F(u;\bm{\nu}_1,\bm{\nu}_2,\bm{\nu}_3) = 0$
\State $(\bm{\nu}_1,\bm{\nu}_2,\bm{\nu}_3) = \argmax\limits_{(\bm{\nu}_1,\bm{\nu}_2,\bm{\nu}_3)\in W_k}F(u;\bm{\nu}_1,\bm{\nu}_2,\bm{\nu}_3)$.
\State $W_{k+1} = \left\{(\bm{\mu}_1,\bm{\mu}_2,\bm{\mu}_3) \mid \bm{\mu}_1\in T_{k+1}^{(1)}(\bm{\nu}_1), \bm{\mu}_2\in T_{k+1}^{(2;\bm{\mu}_1)}(\bm{\nu}_2), \bm{\mu}_3 = \bm{\mu}_1\times\bm{\mu}_2 \right\}$.
\EndFor
\State $u$ $\Leftarrow$ solution of $\quad \max\limits_{(\bm{\nu}_1,\bm{\nu}_2,\bm{\nu}_3)\in W_{k^*}}F(u;\bm{\nu}_1,\bm{\nu}_2,\bm{\nu}_3)$.
\end{algorithmic}
\end{algorithm}

\subsection{Solution methods}
We also describe the techniques used in this implementation to solve the discrete system of nonlinear equations arising from our approximations.  
	In three dimensions, it is not practical to explicitly build Jacobian matrices due to their prohibitively large sizes.  Moreover, many of the PDEs we consider are degenerate and/or have singular solutions.  Thus, Newton's method is not immediately suitable for these problems.

Our approach here is to use a combination of an active set approach~\cite{Bokanowski_Howard}, which has excellent stability properties for the nonlinear systems we consider, and a simple Gauss-Seidel iteration.  Depending on the particular PDE of interest, the solver may collapse into only one of these methods or it may involve a combination.
In the future, this approach could be accelerated using a nonlinear multigrid method.
	
Recall that we are trying to solve systems of the form
\bq\label{eq:solverProb} \max\limits_{\bm{\nu}\in\Af} F(\xvec_i,u_i,\Dt_{\bm{\nu}} u_i, \Dt_{\bm{\nu}\bm{\nu}} u_i) = 0, \quad \xvec_i \in \G. \eq
The same approach works if the maximum above is replaced by a minimum.

The basic approach is to iterate through a two step process.  First, for a given input $u$, we identify the directions $\bm{\nu}_i\in\Af$ that maximize~\eqref{eq:solverProb} at each point in the domain.  Secondly, we fix this direction and seek an approximate solution of
\[F(\xvec_i,u_i,\Dt_{\bm{\nu}_i} u_i, \Dt_{\bm{\nu}_i\bm{\nu}_i} u_i) = 0.\]

In order to implement the second part of the procedure, we recall that our finite difference systems can be written in the form
\[ F(\xvec_i,u_i,\Dt_{\bm{\nu}} u_i, \Dt_{\bm{\nu}\bm{\nu}} u_i) = G_{\bm{\nu}}(\xvec_i,u_i,u_j) \]
where the $\xvec_j$ are points designated as neighbors of $\xvec_i$.  We design a Gauss-Seidel iteration for this by solving for the reference value $u_i$ in terms of the values at neighboring grid points.  That is, we identify a function $G^{-1}_{\bm{\nu}}(\xvec_i,u_j)$ such that
\[ G_{\bm{\nu}}(\xvec_i,G^{-1}_{\bm{\nu}}(\xvec_i,u_j),u_j) = G_{\bm{\nu}}(\xvec_i,u_i,u_j). \]
In some cases (for linear or simple nonlinear operators), this function $G^{-1}_{\bm{\nu}}$ can be identified explicitly.  For example if
\[ G(\xvec_i,u_i,u_j) = -\frac{u_{i+1}+u_{i-1}-2u_i}{h^2} + f(\xvec_i) \]
approximates a simple second derivative then
\[ G^{-1}(\xvec_i,u_j) = \frac{u_{i+1}+u_{i-1}}{2}-\frac{1}{2}f(\xvec_i)h^2. \]
In more complicated examples, this inverse can be obtained (or approximated) through several iterations of a nonlinear solver such as a scalar Newton's method.

	The resulting solution method is described in Algorithm~\ref{alg:solve}. This solver is simple to implement and memory efficient since there is no need to construct the Jacobian matrix. In practice, we can initialize the method with the solution computed on a less refined grid.
	
	\begin{algorithm} 
	  \caption{Solution method for~\eqref{eq:solverProb}}
	  \begin{algorithmic}[1]
	\While{\texttt{Residual > Tolerance}}
		\For{$\xvec_i \in \G$}
	      \State $\bm{\nu}_i = \argmax\limits_{\bm{\nu}\in\Af}G_{\bm{\nu}}(\xvec_i,u_i,u_j)$.
		\EndFor
		\For{$k=1,\ldots,10$}
			\For{$\xvec_i \in \G$}
				\State $u_i = G_{\bm{\nu}_i}^{-1}(\xvec_i,u_j)$.
			\EndFor
		\EndFor
	\EndWhile
	  \end{algorithmic}
	\label{alg:solve}
	\end{algorithm}

\subsection{Eigenvalue problems} % ?
Our framework can also be used to solve eigenvalue problems involving fully nonlinear elliptic PDEs.  These take the form
\bq\label{eq:eigProb}
\begin{cases}
F(\xvec,D^2u(\xvec)) = c, \quad \xvec \in \Omega\\
H(\xvec,\nabla u(\xvec)) = 0, \quad \xvec \in \partial\Omega\\
u(\xvec_0) = 0.
\end{cases}
\eq
Here the constant $c\in\R$ is unknown \emph{a priori}.  These arise from PDEs that require data to satisfy a solvability condition, which may not be known explicitly, may not be satisfied exactly by noisy data, or may not be satisfied at the discrete level even if the original continuous problem is well-posed. Examples that can be cast in this form include the Neumann problem for Poisson's equation, \MA type equations in optimal transport, and the problem of computing minimal Lagrangian graphs~\cite{brendle2010}.

With minor modification, monotone schemes can be used to correctly compute both the eigenvalue $c$ and the solution $u$~\cite{HL_lagrangian}.  Thus our discretization applies immediately to these problems.  The only modification needed is to add an additional unknown $c$ to the nonlinear system, which is also augmented by an additional equation ($u(\xvec_0)=0$) designed to select a unique solution.

\section{Computational Results}\label{sec:results}
We demonstrate the effectiveness of the method by solving a variety of computational examples including a range of challenging nonlinear PDEs and different boundary conditions. 

\subsection{Linear degenerate equation}
We begin by considering a linear degenerate equation posed on the unit sphere:
\bq\label{eq:linDegen}
\begin{cases}
-u_{\bm{\nu}\bm{\nu}}(x,y,z) = 0, & x^2 + y^2 + z^2 < 1\\
u(x,y,z) = \sin{(2\pi(x-\sqrt{2}y-\sqrt{3}z))}, & x^2 + y^2 + z^2 = 1
\end{cases}
\eq
where $\bm{\nu} = \lp1,-1,\frac{-(\sqrt{3}+\sqrt{6})}{3}\rp$.  The exact solution is
$$u(x,y,z) = \sin{(2\pi(x-\sqrt{2}y-\sqrt{3}z))}.$$
Note that the direction $\bm{\nu}$ is not aligned with any Cartesian grid. For this example, neither the grid aligned scheme we derived nor any other grid aligned scheme can be used for a consistent, monotone approximation~\cite{FroeseMeshfreeEigs,MW_Diff}.
Consequently, the generalized finite difference schemes must be used exclusively.

Convergence results are presented in Figure~\ref{LinDegen} and demonstrate better than the expected $\bO(\sqrt{h})$ accuracy.
At one point, we observe a great jump in accuracy as $h$ is refined slightly.  This is due to the scaling of 
the number of boundary points $n_B \approx n^{1/4}$, which experiences a discrete jump from 2 to 3. The resulting improvement in boundary resolution leads to a corresponding decrease in the angular error $d\phi$ near the boundary.  This is of particular importance in this type of fully non-aligned PDE operator, for which the angular discretization error can easily dominate.

\begin{figure}[htbp]
		\begin{center}
\includegraphics[width=0.6\textwidth]{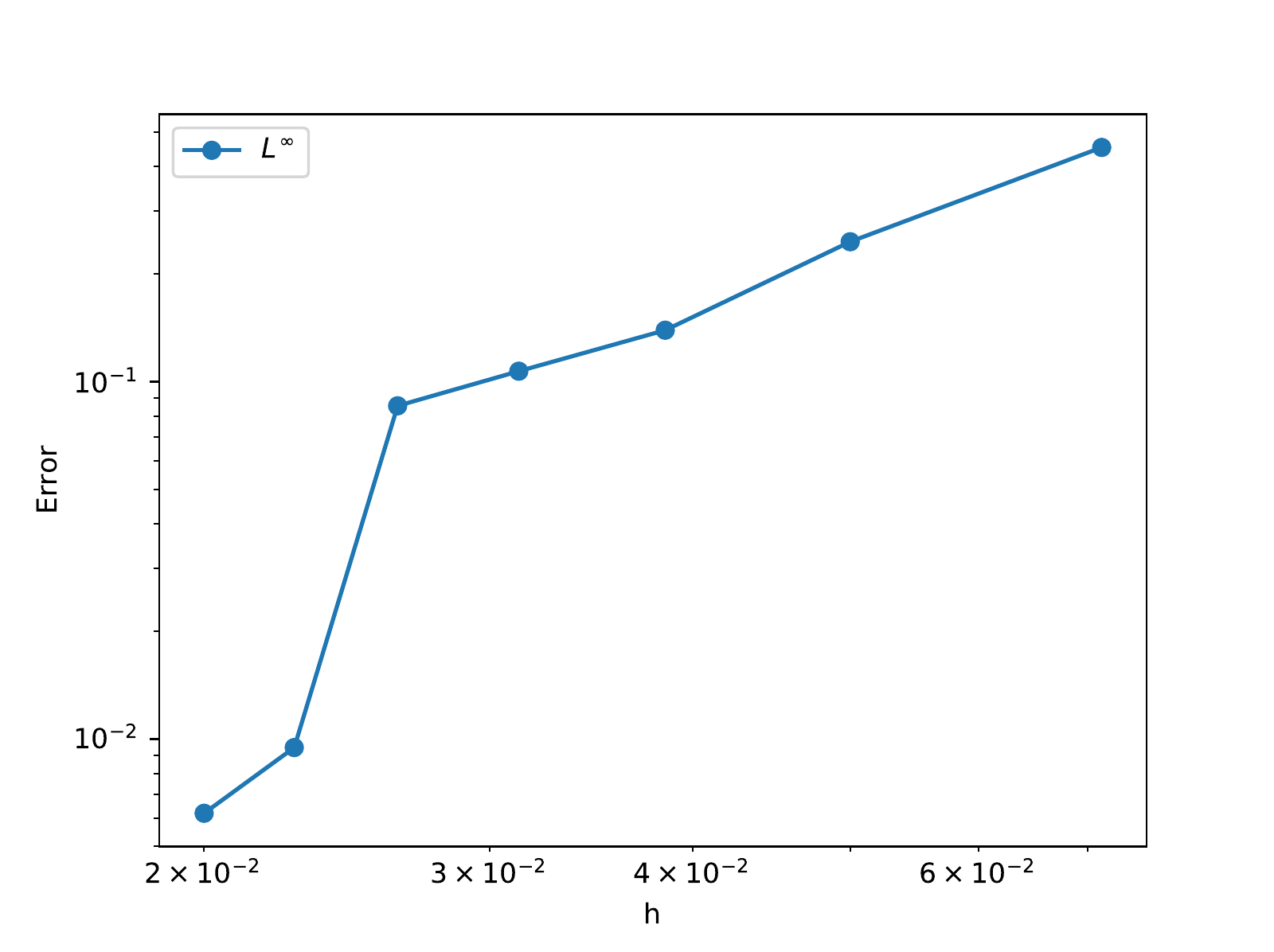}
		\end{center}
\caption{A convergence plot for the linear degenerate equation~\eqref{eq:linDegen}.}
\label{LinDegen}
\end{figure}

\subsection{Maximum of linear operators}
For a second example, consider the following fully nonlinear PDE:
\bq\label{eq:2ops}
\begin{cases}
\max\{-u_{\bm{\nu}_1\bm{\nu}_1},-u_{\bm{\nu}_2\bm{\nu}_2}\} = f(x,y,z) & x^2 + y^2 + z^2 < 1\\
u(x,y,z) = e^{\frac{x^2+y^2+z^2}{2}} & x^2 + y^2 + z^2 = 1
\end{cases}
\eq
where
\[\bm{\nu}_1 = (1,1,0), \quad
\bm{\nu}_2 = (-1,0,1)\]
and
$$f(x,y,z) = \max\left\{-\frac{1}{2}e^{\frac{x^2+y^2+z^2}{2}}(2+x^2+2xy+y^2),-\frac{1}{2}e^{\frac{x^2+y^2+z^2}{2}}(2+x^2-2xz+z^2)\right\}.$$
The exact solution is $$u(x,y,z) = e^{\frac{x^2+y^2+z^2}{2}}.$$

The convergence plot is presented in Figure~\ref{TwoOp}.
As with the linear degenerate equation, there is a discrete jump at one point due to a discrete increase in the boundary resolution. Once again, we observe better accuracy than the expected $\Or(\sqrt{h})$ for this fully nonlinear problem.

\begin{figure}[htbp]
		\begin{center}
\includegraphics[width=0.6\textwidth]{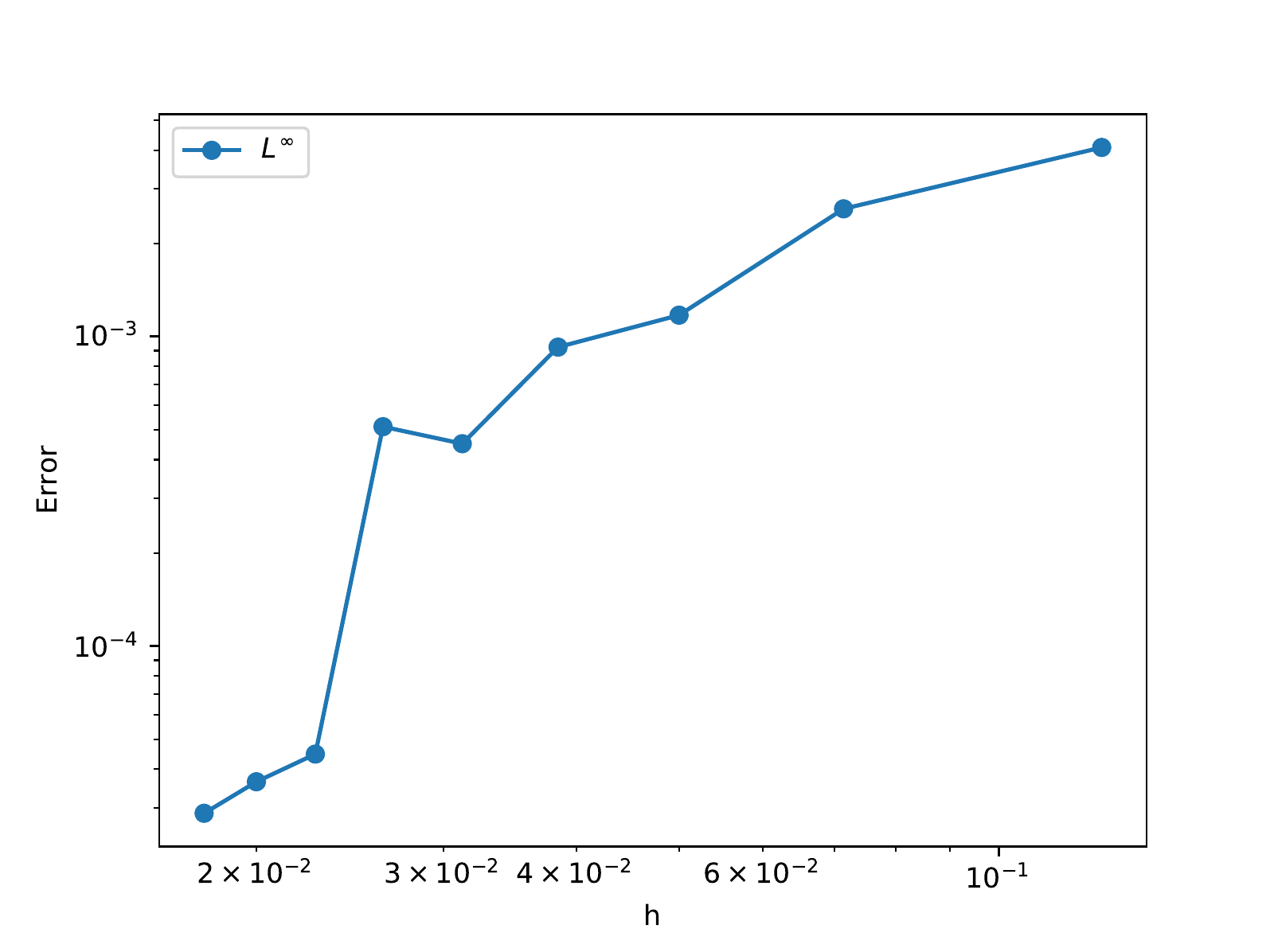}
		\end{center}
\caption{A convergence plot for the two-operator problem~\eqref{eq:2ops}.}
\label{TwoOp}
\end{figure}

\subsection{Convex envelope equation}
Next, we consider a PDE for computing the convex envelope of an obstacle~\cite{ObermanCE}.
\bq\label{eq:CE}
\begin{cases}
\max\{-\la_1(D^2u),u-g\} = 0 & x^2 + y^2 + z^2 < 0.25\\
u = 0.2 &  x^2 + y^2 + z^2 = 0.25
\end{cases}
\eq
where
$$g(x,y,z) =\min\left\{2\sqrt{(x^2+y^2+z^2)},0.2\right\}.$$
In addition to being a fully nonlinear equation, the solution to this PDE is only Lipschitz continuous (but is not differentiable at the origin). Thus, it must be interpreted in a weak sense, and the use of a discretization that converges to the viscosity solution is imperative. The exact solution for this problem is the cone
$$u(x,y,z) = 0.4\sqrt{x^2 + y^2 + z^2}.$$

We remark that this PDE involves only the smallest eigenvalue of the Hessian matrix. It can therefore be characterized using the traditional Rayleigh-Ritz form and discretized as
\[ \lambda_1(D^2u) \approx \min\limits_{\bm{\nu}\in E^h} \Dt_{\bm{\nu}\bm{\nu}}u.  \]

The convergence plot is presented in Figure~\ref{ConvEnv}.
Note that convergence is not monotone in this case.  This is due to effects of variations in the alignment of the grid points for different $n$ (by chance, some small values of $n$ can lead to grids that are very well aligned with the singularity). This effect has previously been observed in two dimensions for problems with very low regularity~\cite{Hamfeldt_Gauss}. Nevertheless, we observe overall convergence close to $\Or(\sqrt{h})$ even on this very singular example.

\begin{figure}[htbp]
		\begin{center}
\includegraphics[width=0.6\textwidth]{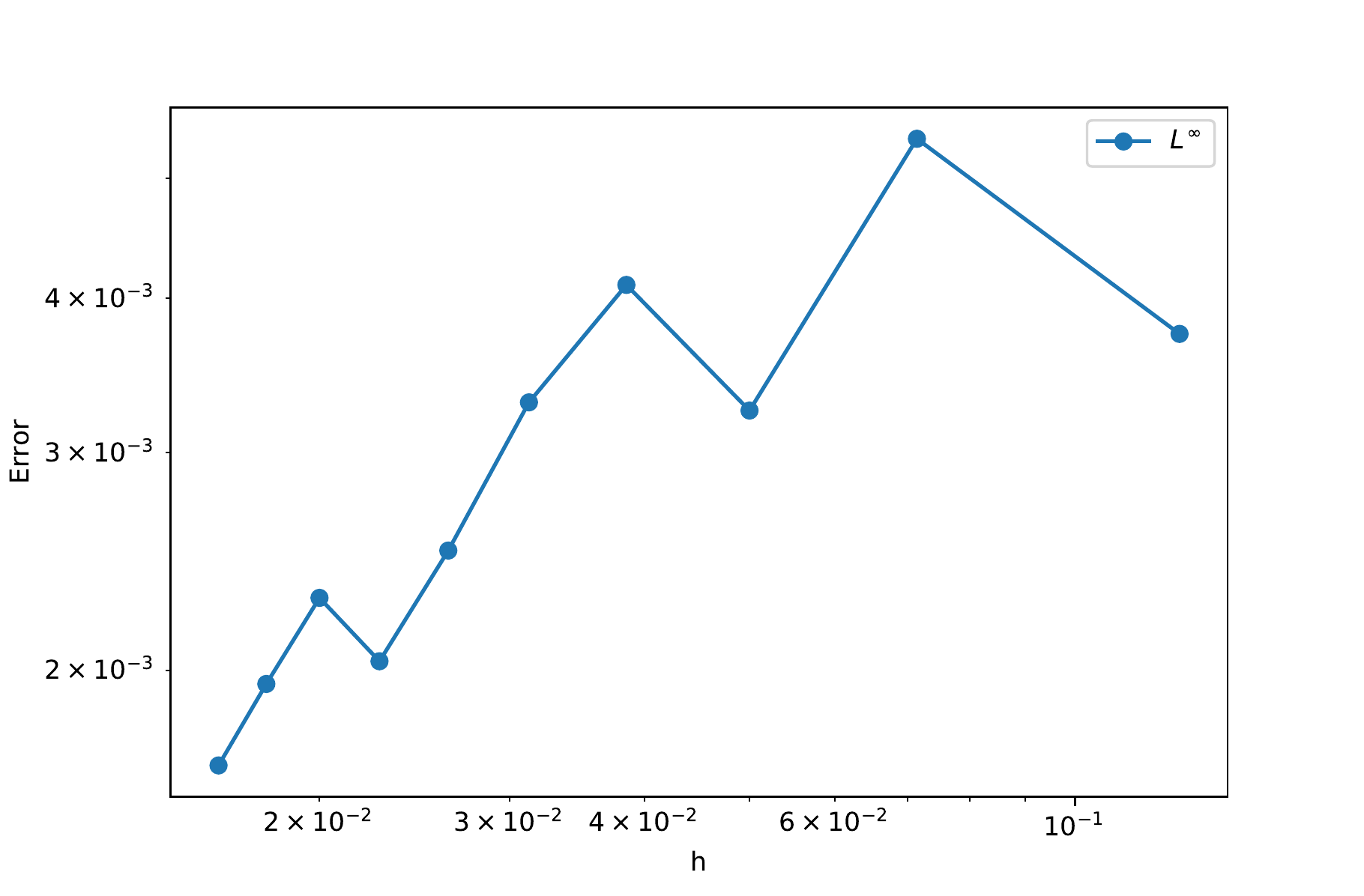}
		\end{center}
\caption{A convergence plot for the convex envelope equation~\eqref{eq:CE}.}
\label{ConvEnv}
\end{figure}

%\subsection{Poisson's Equation}
%For the next three examples, we demonstrate the effectiveness of the method for solving functions of the eigenvalues of the Hessian. First, consider something as simple as Poisson's equation, which can be expressed as 
%\bq
%\la_1(D^2u(x,y,z)) + \la_2(D^2u(x,y,z)) + \la_3(D^2u(x,y,z)) = f(x,y,z)
%\eq
%In order to use an example where the method will not be exact, we choose
%\bq
%\begin{cases}
%\sum_i \la_i(D^2u(x,y,z)) =e^{\frac{x^2 + y^2 + z^2}{2}}(3+x^2+y^2+z^2) & x^2 + y^2 + z^2 < 1\\
%u(x,y,z) = e^{\frac{x^2 + y^2 + z^2}{2}} & x^2 + y^2 + z^2 = 1\\
%\end{cases}
%\eq
%In order to approximate the eigenvalues, we needed to minimize
%$$\min{(u_{\nu_1\nu_1}+u_{\nu_2\nu_2}+u_{\nu_3\nu_3})}$$
%over all frames such that $\nu_1, \nu_2, \nu_3$ are orthogonal. Thus, both the grid aligned scheme and the meshfree scheme are used at various points to approximate a sampling of those frames, and then we take the discrete minimum. The results of this approximation are in Figure \ref{3dPoisson}. The exact solution is 
%$$u(x,y,z) = e^{\frac{x^2+y^2+z^2}{2}}$$
%
%\begin{figure}[htbp]
		%\begin{center}
%\includegraphics[width=\figwidth]{PoissonCircleExp}
		%\end{center}
%\caption{A convergence plot for Poisson's Equation on a sphere.}
%\label{3dPoisson}
%\end{figure}

\subsection{\MA equation}
Next, we turn our attention to more general functions of the eigenvalues of the Hessian matrix.  We begin with the \MA equation:
\[\begin{cases}
-\det{(D^2u(\xvec))}+f(\xvec) = 0,& \xvec \in \Omega\\
u(x) = g(\xvec) & \xvec\in \partial\Omega\\
u \text{ is convex.} &
\end{cases}
\]
The determinant can be expressed as a product of the eigenvalues.
Since the equation is only elliptic on the space of convex functions, we follow~\cite{FroeseTransport} and use the globally elliptic extension
\bq
-\max{(\la_1,0)}\max{(\la_2,0)}\max{(\la_3,0)} - \left(\min{(\la_1,0)} + \min{(\la_2,0)} + \min{(\la_3,0)}\right)  + f = 0.
\eq
We notice that this can be decomposed into two different functions of the eigenvalues, each of which can be written in the form of~\eqref{eq:PDE2}.  That is, let $\phi_1(x) = \log\max\{x,0\}$, $G_1(x) = -e^x$, $\phi_2(x)=\min\{x,0\}$, and $G_2(x) = -x$.  Then we can re-express this \MA equation as
\[ G_1\left(\sum\limits_{j=1}^3\phi_1(D^2u(\xvec))\right) + G_2\left(\sum\limits_{j=1}^3\phi_2(D^2u(\xvec))\right) + f(\xvec) = 0, \]
similar to~\cite{FO_MATheory}.  This now fits within the framework we require to produce consistent, monotone approximations.

Consider the specific example
\bq\label{eq:MA}
\begin{cases}
-\det{(D^2u(x,y,z))}+e^{\frac{3}{2}(x^2 + y^2 + z^2)}(1+x^2+y^2+z^2)= 0,& x^2+y^2+z^2 < .25\\
u(x,y,z) = e^{\frac{x^2+y^2+z^2}{2}} & x^2+y^2+z^2 = .25\\
u \text{ is convex.} &
\end{cases}
\eq
with the exact solution being
$$u(x,y,z) = e^{\frac{x^2+y^2+z^2}{2}}.$$

The results are included in Figure \ref{MAfig}.  On this example, we also observe better than the expected $\bO(\sqrt{h})$ convergence.

\begin{figure}[htbp]
		\begin{center}
\includegraphics[width=0.6\textwidth]{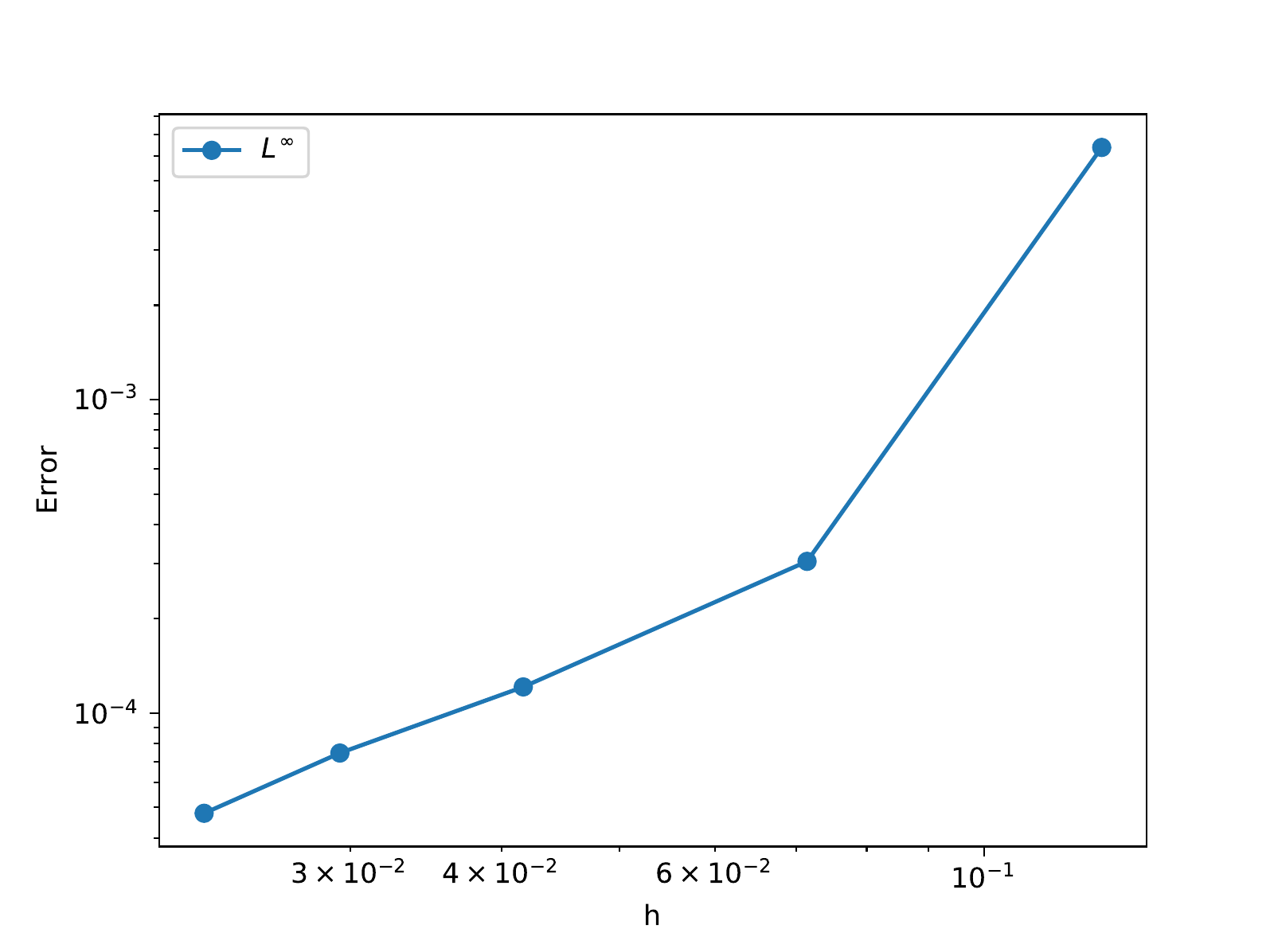}
		\end{center}
\caption{A convergence plot for the \MA equation~\eqref{eq:MA}.}
\label{MAfig}
\end{figure}

%\subsection{Lagrangian Curvature Problem}
%Finally, we consider a variation of the Lagrangian Curvature problem from \cite{brendle2010}. Specifically, we solve
%\bq
%\begin{cases}
%\itan{\la_1(D^2u)} + \itan{\la_2(D^2u)}+\itan{\la_3(D^2u)} = f(x,y,z) & x^2 + y^2 + z^2 < 1\\
%u(x,y,z) = e^{\frac{x^2+y^2+z^2}{2}} & x^2 + y^2 + z^2 = 1
%\end{cases}
%\eq
%where 
%$$f(x,y,z) = 2\itan{e^{\frac{x^2+y^2+z^2}{2}}}+\itan{(1+x^2+y^2+z^2)e^{\frac{x^2+y^2+z^2}{2}}}$$
%The exact solution is $$u(x,y,z) = e^{\frac{x^2+y^2+z^2}{2}}$$
%This is also a fully nonlinear PDE. As we proved in \autoref{sec:Hessianevals}, we can create a consistent monotone approximation of this operator by minimizing the directional derivatives over all orthogonal frames. Thus, we approximate the operator by
%$$\min_{\nu_1,\nu_2,\nu_3 \text{ orthogonal}}(\itan{u_{\nu_1\nu_1}} + \itan{u_{\nu_2\nu_2}}+\itan{u_{\nu_3\nu_3}})$$
%where the $u_{\nu_i\nu_i}$ are approximated by the monotone stencils we derived earlier.
%As with the other problems, we approximate the directional derivatives for a subset of orthogonal frames, then take a discrete minimum. The results are in Figure \ref{3datan}.
%\begin{figure}[htbp]
		%\begin{center}
%\includegraphics[width=\figwidth]{ArctanExpConv}
		%\end{center}
%\caption{A convergence plot for the Lagrangian Curvature Equation.}
%\label{3datan}
%\end{figure}
%We observe accuracy ranging from the expected $\Or(\sqrt{h})$ to $\Or(h^2)$.

\subsection{Neumann boundary conditions}
%Next, we consider a variation of the Lagrangian Curvature problem from \cite{brendle2010} with Neumann boundary conditions. Specifically, we solve
%\bq
%\begin{cases}
%\itan{\la_1(D^2u)} + \itan{\la_2(D^2u)}+\itan{\la_3(D^2u)} = c\:f(x,y,z) & x^2 + y^2 + z^2 < 1\\
%\frac{\partial u(x,y,z)}{\partial \hat{n}} = e^{\frac{1}{2}} & x^2 + y^2 + z^2 = 1
%\end{cases}
%\eq
%where 
%$$f(x,y,z) = 2\itan{e^{\frac{x^2+y^2+z^2}{2}}}+\itan{(1+x^2+y^2+z^2)e^{\frac{x^2+y^2+z^2}{2}}}$$
%The exact solution is $$u(x,y,z) = e^{\frac{x^2+y^2+z^2}{2}}$$
%As with the other problems, we approximate the second directional derivatives for a subset of orthogonal frames, then take a discrete minimum. We also approximate the normal derivatives using the scheme derived in \autoref{sec:boundarycond}. The results are in Figure \ref{3datanNeumann}.
%\begin{figure}[htbp]
		%\begin{center}
%\includegraphics[width=\figwidth]{ArctanExpNeumannConvergenceMonotone}
		%\end{center}
%\caption{A convergence plot for the Lagrangian Curvature Equation with Neumann boundary conditions.}
%\label{3datanNeumann}
%\end{figure}
%We observe close to the expected $\Or(\sqrt{h})$ accuracy on this example.

Next, we consider Poisson's equation with Neumann boundary conditions. 
The point of this example is, of course, not to produce a new method for solving Poisson's equation.  Instead, we use it to test our characterization of functions of the eigenvalues of the Hessian, our generalized finite difference implementation of Neumann boundary conditions, and our ability to solve eigenvalue problems.

We recall that the data must satisfy a solvability condition in order for a solution to exist.  Moreover, even if the continuous problem is well-posed, the discretized problem need not be~\cite{HL_lagrangian}. Therefore, we choose to frame this as the following eigenvalue problem:
\bq\label{eq:neumann}
\begin{cases}
-\left(\la_1(D^2u) + \la_2(D^2u)+\la_3(D^2u)\right) = cf(x,y,z) & x^2 + y^2 + z^2 < 1\\
\frac{\partial u(x,y,z)}{\partial \hat{n}} = e^{\frac{1}{2}} & x^2 + y^2 + z^2 = 1
\end{cases}
\eq
where 
$$f(x,y,z) =(3+x^2+y^2+z^2)e^{\frac{x^2+y^2+z^2}{2}}.$$
The exact solution is $$u(x,y,z) = e^{\frac{x^2+y^2+z^2}{2}}$$ with $c = 1$.

To further test our characterization of functions of the eigenvalues of the Hessian, we also note that the Laplacian does trivially have the form of~\eqref{eq:PDE2} with $\phi(x) =x$ and $G(x) = -x$.
The results are presented in Figure \ref{3dPoissonNeumann}.  On average we observe the expected $\Or(\sqrt{h})$ accuracy on this example.

\begin{figure}[htbp]
		\begin{center}
\includegraphics[width=0.6\textwidth]{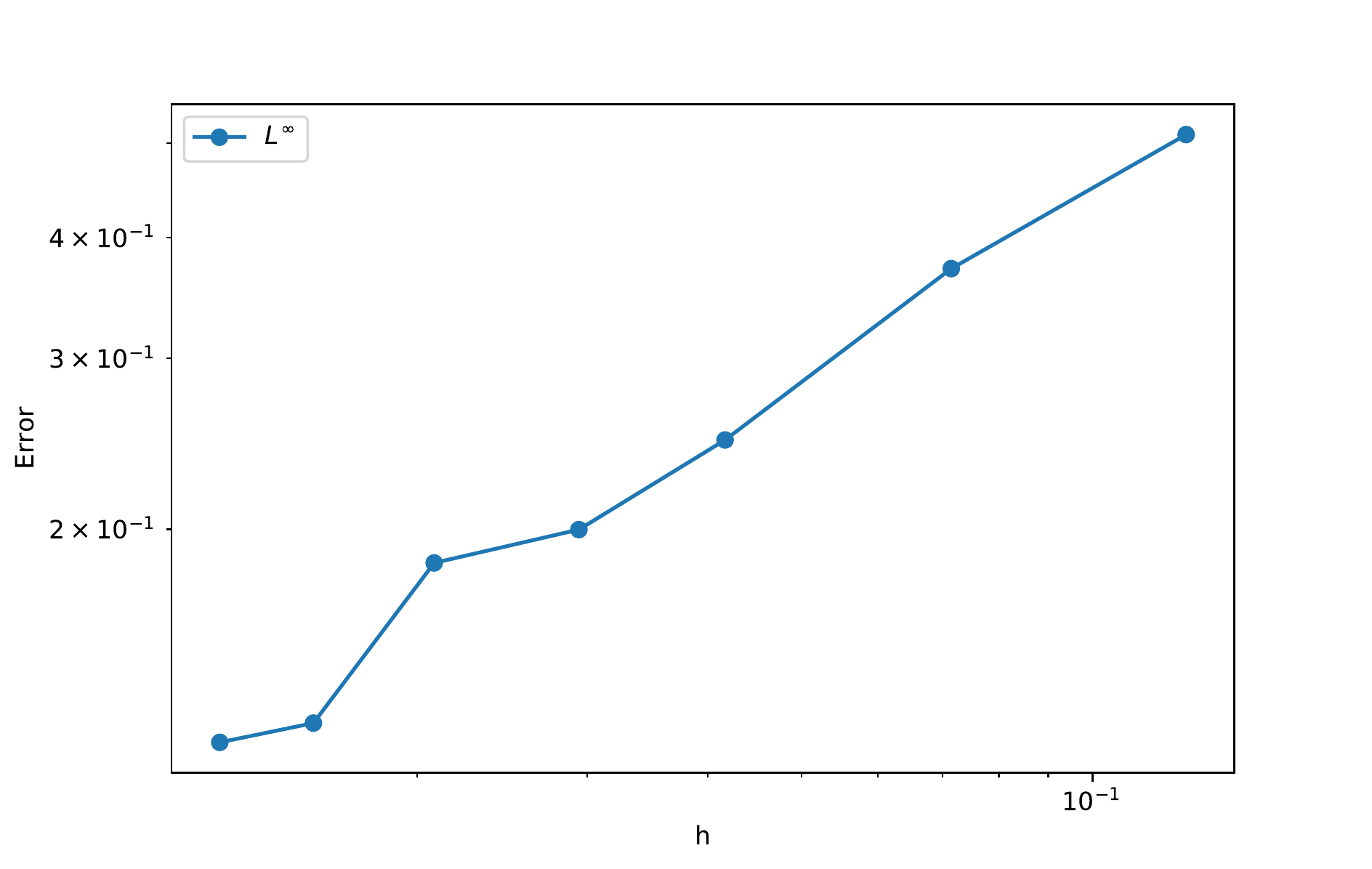}
		\end{center}
\caption{A convergence plot for Poisson's equation with Neumann boundary conditions~\eqref{eq:neumann}.}
\label{3dPoissonNeumann}
\end{figure}

\subsection{Second type boundary conditions}
Finally, we consider the problem of computing minimal Lagrangian graphs~\cite{brendle2010}.  This is an eigenvalue problem for a nonlinear PDE, equipped with the second type (optimal transport) boundary condition.
Specifically, we seek a convex function $u$ and a scalar constant $c\in\R$ satisfying 
\bq\label{eq:lagrange}
\begin{cases}
\itan{\la_1(D^2u)} + \itan{\la_2(D^2u)}+\itan{\la_3(D^2u)} = c, & x^2 + y^2 + z^2 < 1\\
\nabla u(\mathbb{S}^2) \subset T(\mathbb{S}^2).
\end{cases}
\eq
where $T(x,y,z) = (x+2,y+1,z-1)$ is an affine shift. 
The exact solution is $$u(x,y,z) = \frac{(x+2)^2+(y+1)^2+(z-1)^2}{2}.$$

In order to discretize this function of the eigenvalues of the Hessian, we first need to put the PDE operator into the form of~\eqref{eq:PDE2}.  To accomplish this, we introduce a modification that agrees with~\eqref{eq:lagrange} on the space of convex functions (which is where the desired solution lives).  In particular, we propose the alternate operator
\[ -\sum\limits_{j=1}^3\left(\itan{\max\{\lambda_j,0\}} + \min\{\lambda_j,0\}\right), \]
which fits immediately into the required form.  The optimal transport type boundary constraint is discretized as described in \autoref{sec:BVP2}.

The computed results are shown in Figure~\ref{3datanTP}. Once again, we observe slightly better than the expected $\Or(\sqrt{h})$ accuracy overall.

\begin{figure}[htbp]
		\begin{center}
\includegraphics[width=0.6\textwidth]{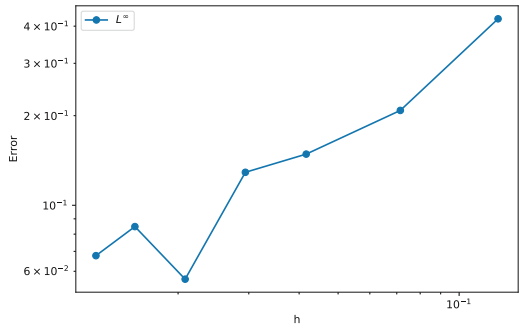}
		\end{center}
\caption{A convergence plot for the computation of minimal Lagrangian graphs~\eqref{eq:lagrange}.}
\label{3datanTP}
\end{figure}

\section{Conclusion}\label{sec:conclusion}
%Conclusion:
In this paper, we introduced a new monotone finite difference method for solving a wide variety of fully nonlinear elliptic equations in three dimensions.  Because the resulting schemes are monotone, they are guaranteed to converge via the Barles-Souganidis convergence framework~\cite{BSnum} and generalizations of these to non-classical Dirichlet problems~\cite{Hamfeldt_Gauss}, optimal transport problems~\cite{Hamfeldt_BVP2}, and eigenvalue problems involving nonlinear PDEs~\cite{HL_lagrangian}.

In particular, we described a new technique for discretizing general three-dimensional domains, which produces the higher boundary resolution needed to preserve both consistency and monotonicity throughout the entire domain.  We also introduced and analyzed a simple least-squares method for generating consistent, monotone approximations of second directional derivatives.  Moreover, we showed how to use these to construct monotone approximations of a large range of fully nonlinear elliptic operators.  Finally, we produced generalized finite difference approximations for a range of different boundary conditions including Dirichlet, Neumann, and the nonlinear second type (optimal transport) boundary condition.

This paper focused primarily on efficiently constructing these approximations in three dimensions, which is much more expensive than the analogous problem in two dimensions.  Because our grids inherit much of the structure of a Cartesian grid, constructing the finite difference stencils is straightforward throughout most of the domain.  A more serious computational challenge in three dimensions is evaluating nonlinear operators that require computing a maximum/minimum over many different orthogonal coordinate frames; this is needed for many fully nonlinear operators.  We proposed a multilevel approach to this evaluation of the nonlinear operators, which converted the problem from one that is completely intractable in 3D to a very efficient process.

In future work, we intend to leverage this new scheme, which can be evaluated very efficiently, to produce numerical methods that are both efficient and higher-order.  In particular, we will develop faster solvers for the discretized systems that utilize the underlying structure of the monotone approximations.  We will also use these as a foundation for convergent, higher-order filtered methods~\cite{FOFiltered}.

\bibliographystyle{plain}
\bibliography{PaperBibliography}
\end{document}